\documentclass[letterpaper,10pt,twocolumn,conference]{ieeeconf}
\usepackage{amsmath,amssymb,euscript,psfrag,latexsym,graphicx}
\usepackage{bbm,color,amstext,wasysym,subfig,parskip,balance}
\usepackage{cite}
\usepackage{amsfonts,mathrsfs,mathtools}
\usepackage{bm}
\usepackage{slashbox}
\usepackage{soul}

\newcommand*\Laplacian{\mathop{}\!\mathbin\bigtriangleup}

\newtheorem{lemma}{Lemma}
\newtheorem{corollary}{Corollary}

\newcommand{\mR}{{\mathbb R}}

\newcommand{\cX}{{\mathcal X}}
\newcommand{\cY}{{\mathcal Y}}
\newcommand{\cD}{{\mathscr{D}}}
\newcommand{\cS}{{\mathcal{S}}}

\newcommand{\bbc}{\bm{c}}

\newcommand{\bbv}{\bm{v}}
\newcommand{\bbw}{\bm{w}}
\newcommand{\bbx}{{\bm{x}}}
\newcommand{\bby}{{\bm{y}}}
\newcommand{\bbz}{\bm{z}}

\newcommand{\sym}{^{\mathrm{sym}}}
\newcommand{\skewsym}{^{\mathrm{skew}}}
\newcommand{\ep}{_{\mathrm{ep}}}

\newcommand{\bbmu}{\bm{\mu}}
\newcommand{\bbP}{\bm{P}}
\newcommand{\bbK}{\bm{K}}
\newcommand{\bbR}{\bm{R}}
\newcommand{\bbA}{\bm{A}}
\newcommand{\bbB}{\bm{B}}
\newcommand{\bbC}{\bm{C}}
\newcommand{\bbI}{\bm{I}}
\newcommand{\bbL}{\bm{L}}
\newcommand{\bbF}{\bm{F}}
\newcommand{\bbG}{\bm{G}}
\newcommand{\bbQ}{\bm{Q}}

\newcommand{\tr}{\operatorname{tr}}
\newcommand{\argmin}{\operatorname{arg\:min}}
\newcommand{\arginf}{\operatorname{arg\:inf}}

\graphicspath{{./},{./figures/}}
\def\spacingset#1{\def\baselinestretch{#1}\small\normalsize}
\spacingset{1}

\begin{document}

\title{Gradient Flows in Uncertainty Propagation\\ and Filtering of Linear Gaussian Systems}
\author{Abhishek Halder, and Tryphon T. Georgiou\\{\em University of California, Irvine}}
\maketitle

\begin{abstract}
The purpose of this work is mostly expository and aims to elucidate the Jordan-Kinderlehrer-Otto (JKO) scheme for uncertainty propagation, and a variant, the Laugesen-Mehta-Meyn-Raginsky (LMMR) scheme for filtering. We point out that these variational schemes can be understood as proximal operators in the space of density functions, realizing gradient flows. These schemes hold the promise of leading to efficient ways for solving the Fokker-Planck equation as well as the equations of non-linear filtering. Our aim in this paper is to develop in detail the underlying ideas in the setting of linear stochastic systems with Gaussian noise and recover known results.
\end{abstract}

\section{Introduction}
Consider the gradient flow
$
\frac{\mathrm{d}\bbx}{\mathrm{d}t}=-\nabla \psi(\bbx)
$
in $\mR^n$, where $\nabla$ is the gradient (w.r.t. the Euclidean metric) of a function $\psi(\bbx)$, and consider the discretization
\[
\bbx_k=\bbx_{k-1}-h \nabla \psi(\bbx_{k-1}), \mbox{ for }k\in\mathbb{N}.
\]
As is well known in finite-dimensional optimization,
\begin{align}
\bbx_{k} &= \underset{\bbx}{\argmin}\{ \frac12\|\bbx - (\bbx_{k-1}-h \nabla \psi(\bbx_{k-1}))\|^2\}\nonumber\\
%&=\underset{\bbx}{\argmin} \{\frac{1}{2}\|\bbx-\bbx_k\|^2+h\langle (\bbx-\bbx_k),\nabla \psi(\bbx_k)\rangle\}\nonumber\\
%&\hspace*{-.7cm}=\underset{\bbx}{\argmin}  \{\frac{1}{2h}\|\bbx-\bbx_k\|^2+\langle (\bbx-\bbx_k),\nabla \psi(\bbx_k)\rangle+\psi(\bbx_k)\}\nonumber\\
&=\underset{\bbx}{\argmin} \frac{1}{2}\|\bbx-\bbx_{k-1}\|^2+ h\psi(\bbx) +o(h). \label{eq:jko}
\end{align}
By recursively evaluating the proximal operator \cite{BauschkeCombettes2011, ParikhBoyd2013}
\begin{eqnarray*}
\bbx_{k}&=&{\rm prox}_{h\psi}^{\|\cdot\|}(\bbx_{k-1})\\&=&\underset{\bbx}{\argmin} \{\frac{1}{2}\|\bbx-\bbx_{k-1}\|^2+h\psi(\bbx)\},
\end{eqnarray*}
the solution, which depends on the choice of the step size $h$, satisfies $\bbx_k(h)\to \bbx(t=kh)$, as $h\to 0$.

The Jordan-Kinderlehrer-Otto (JKO) scheme, introduced in \cite{JKO1998}, is a similar recursion in the infinite-dimensional space of density functions with respect to the Wasserstein geometry \cite{VillaniBook2003}, namely,
\begin{align}
\varrho_{k}(\bbx,h)&=\underset{\varrho}{\argmin}\;  \frac12 W_2^2(\varrho,\varrho_{k-1}) + h\cS(\varrho), \;\; k\in\mathbb{N},
\label{eq:JKO}\end{align}
where $W_2(\cdot,\cdot)$ denotes the Wasserstein-2 distance between two (probability) density functions,
\begin{equation}\label{eq:entropy}
\cS(\varrho):=\int_{\mR^n}\varrho(\bbx)\log(\varrho(\bbx))\mathrm{d}\bbx
\end{equation}
is the negative differential entropy functional,
and $\mathrm{d}\bbx$ is the volume element. In other words, \eqref{eq:JKO} can be viewed as the proximal operation ${\rm prox}_{h\cS}^{W_2}(\varrho_{k-1})$.
The main result in \cite{JKO1998} was to show that the minimizer of \eqref{eq:JKO} approximates the solution $\rho(\bbx,t)$ of the heat equation \[
\frac{\partial\rho(\bbx,t)}{\partial t}= \Delta \rho(\bbx,t), \mbox{ with }\rho(\bbx,0)=\rho_0(\bbx),
\]
in the sense that
$
\varrho_{k}(\bbx,h)\to \rho(\bbx,t=kh),\mbox{ as }h\downarrow 0.
$
Thus, \eqref{eq:JKO} establishes the remarkable result that {\em the heat equation is the gradient descent flow of the (negative) entropy integral with respect to the Wasserstein metric}.

An analogous JKO-like scheme was introduced recently in Laugesen \emph{et al.} \cite{LaugesenMehta2015} for the measurement update-step in continuous-time filtering. More specifically, let us consider the general system of It\^{o} stochastic differential equations (SDE's)
\begin{subequations}
\begin{eqnarray}
\mathrm{d}\bm{x}(t) &=& -\nabla U(\bm{x})\:\mathrm{d}t + \sqrt{2\beta^{-1}}\:\mathrm{d}\bm{w}(t),
\label{jkoSDE}\\
\mathrm{d}\bbz(t) &=& \bbc(\bm{x}(t),t)\:\mathrm{d}t + \mathrm{d}\bbv(t), 
\label{ObservationSDE}	
\end{eqnarray}
\end{subequations}
where $\bbx\in\mathbb{R}^{n}, \bbz\in\mathbb{R}^{m}, \beta>0$, $U(\cdot)$ is a potential, the process and measurement noise processes $\bbw(t)$ and $\bbv(t)$ are Wiener and satisfy $\mathbb{E}\left[\mathrm{d}w_{i}\mathrm{d}w_{j}\right] = \bbQ_{ij}\mathrm{d}t \:\forall\:i,j=1,\hdots,n$ and $\mathbb{E}\left[\mathrm{d}v_{i}\mathrm{d}v_{j}\right] = \bbR_{ij}\mathrm{d}t \:\forall\:i,j=1,\hdots,m$, with $\bbQ,\bbR \succ \bm{0}$, respectively. Then $\bbx(t)$ and $\bbz(t)$ represent state and sensor measurements at time $t$. Further, as usual, $\bbv(t)$ is assumed to be independent of  $\bbw(t)$ and independent of the initial state $\bm{x}(0)$.
Given the history of noise corrupted sensor data up to time $t$, the filtering problem requires computing  
the posterior probability distribution that obeys the Kushner-Stratonovich stochastic PDE \cite{Stratonovich1960, Kushner1964,FKK1972}.

For the special case of trivial state dynamics, i.e., $\mathrm{d}\bbx=0$, and $\bbR$ the identity,
Laugesen \emph{et al.} \cite{LaugesenMehta2015} introduced
\begin{align}
\varrho_{k}^{+}(\bbx,h) =  \underset{\varrho\in\mathcal{D}_{2}}\arginf  \{ D_{\mathrm{KL}}\left(\varrho\|\varrho_{k}^{-}\right) + h\Phi(\varrho)\}, \: k\in\mathbb{N},
\label{KLvariational}	
\end{align}
with
\begin{align}\label{eq:Phi}
\Phi(\varrho) &:= \frac12\mathbb{E}_{\varrho}\{(\bby_{k} - \bbc(\bbx))^{\top}\bbR^{-1}(\bby_{k} - \bbc(\bbx))\},
\end{align}
where $\bby_{k}$ is the  noisy measurement in discrete-time
defined via 
$\bby_{k} := \frac{1}{h}\Delta \bbz_{k}$, $\Delta\bbz_{k}:=\bbz_{k} - \bbz_{k-1}$, and 
$\{\bbz_{k-1}\}_{k\in\mathbb{N}}$ the sequence of samples of $\bbz(t)$ at $\{t_{k-1}\}_{k\in\mathbb{N}}$ for
$t_{k-1} := (k-1)h$.
Laugesen \emph{et al.} \cite{LaugesenMehta2015} proved that
the {\em LMMR equation} (\ref{KLvariational}) approximates the solution of 
\begin{align}
\mathrm{d}\rho^{+}(\bbx(t),t) =& \left[ \left(\bbc(\bbx(t),t) - \mathbb{E}_{\rho^{+}}\{\bbc(\bbx(t),t)\}\right)^{\top} \bbR^{-1} \right.\nonumber\\
&\!\!\!\!\!\!\!\!\!\!\left.\left(\mathrm{d}\bbz(t) - \mathbb{E}_{\rho^{+}}\{\bbc(\bbx(t),t)\}\mathrm{d}t\right)\right]\: \rho^{+}(\bbx(t),t),
\label{KSpde}	
\end{align}
i.e., of the Kushner-Stratonovich PDE corresponding to $\mathrm{d}\bbx=0$,  in the sense that 
$\varrho_{k}^{+}(\bm{x},h)\rightharpoonup \rho^{+}(\bbx(t),t)$ over $t\in[(k-1)h, kh)$, as $h\downarrow 0$.  Thus, they showed that in this special case, {\em the Kushner-Stratonovich PDE is the gradient descent of functional $\Phi(\cdot)$ with respect to $D_{\rm KL}$, i.e., computed by ${\rm prox}_{h\Phi}^{D_{\rm KL}}(\varrho_k^-)$}.

The purpose of the present paper is to develop this circle of ideas, namely, that {\em both uncertainty propagation and filtering can be viewed as gradient flows} in the special case of linear stochastic systems with Gaussian noise. In fact, we consider the general case of
the linear stochastic system
\begin{eqnarray}
\mathrm{d}\bm{x}(t) = \bm{A}\bm{x}(t) \: \mathrm{d}t\: + \bm{B}\:\mathrm{d}\bm{w}(t),
\label{OUmultivariate}
\end{eqnarray}
where $\bm{w}(t)$ is a Wiener process as before, though possibly not of the same dimension as $\bbx$. We suppose that the uncertain initial condition $\bm{x}(0)$ has a known Gaussian PDF, the matrix $\bm{A}$ is Hurwitz, and that the diffusion matrix $\bm{B}$ is such that $(\bm{A},\bm{B})$ is a controllable pair. For this, we recover the well-known propagation equations (see for example \cite[Ch. 3.6]{AstromBook1970}) for the mean and covariance of the state $\bbx(t)$ out of the JKO-scheme via a \emph{two-step optimization}. The applicability of the JKO-scheme to \eqref{OUmultivariate} {\em is not immediately obvious} since the development in \cite{JKO1998} requires the state dynamics to be in the canonical form \eqref{jkoSDE} with the drift being a gradient and the diffusion coefficient being a positive scalar. We further show that this two-step optimization procedure that we introduce, can be used to derive the Kalman-Bucy filter from a generalized version of the LMMR equation (\ref{KLvariational}). We remark that variational schemes for estimator/observer design based on gradient flows can also be seen as regularized dynamic inversion in the spirit of \cite{YezziVerriest2007}.

%\blue{Original technical contributions of this paper include}

\subsection*{Notation}
Throughout we use bold-faced upper-case letters for matrices, and bold-faced lower case letters for vectors. The notation $\bm{I}$ stands for identity matrix of appropriate dimension, we use $\tr(\cdot)$ and $\det(\cdot)$ to respectively denote the trace and determinant of a matrix, and the symbols $\nabla$ and $\Laplacian$ denote the gradient and Laplacian operators, respectively. We denote the space of probability density functions (PDFs) on $\mR^n$ by $\cD := \{\rho: \rho \geq 0, \int_{\mR^n} \rho = 1\}$, by $\cD_2:=\{\rho\in\cD \mid \int_{\mR^n} \bm{x}^{\top}\bm{x}\:\rho(\bm{x})\mathrm{d}\bm{x} < \infty\}$ the space of PDFs with finite second moments, by $\cD_{\bm{\mu},\bm{P}}$ denote the space of PDFs which share the same mean vector $\bm{\mu}$ and same covariance matrix $\bbP:=\int_{\mR^n}(\bbx-\bm{\mu})(\bbx-\bm{\mu})^\top \rho(\bbx)\mathrm{d}\bbx$. Likewise, let $\cD_{\bm{\mu},\tau}$ denote the space of PDFs which have the same mean $\bm{\mu}$ and same trace of covariance $\tau:=\tr(\bbP)>0$. Clearly, $\cD_{\bm{\mu},\bm{P}} \subset \cD_{\bm{\mu},\tau} \subset \cD_{2} \subset \cD$. We use the symbol $\mathcal{N}\left(\bm{\mu},\bm{P}\right)$ to denote a multivariate Gaussian PDF with mean $\bm{\mu}$, and covariance $\bm{P}$. The notation $\bbx\sim\rho$ means that the random vector $\bbx$ has PDF $\rho$; and $\mathbb{E}\left\{\cdot\right\}$ denotes the expectation operator while, when the probability density is to be specified, $\mathbb{E}_\rho\left\{\cdot\right\}:=\int_{\mR^n}(\cdot)\rho(\bbx)\mathrm{d}\bbx$.

%%%%%%%%%%%%%%%%%%%%%%%%%%%%%%%%%%%%%%%%%%%%%%%%%%%%%%%%%%%%%%%%%%%%%%%%%%%%%%%

\section{JKO Scheme in General}

We now discuss in some detail the JKO scheme for the case of the diffusion process in \eqref{jkoSDE}, and the corresponding Fokker-Planck equation \cite{RiskenBook1989}
\begin{equation}\label{eq:FP}
\frac{\partial\rho}{\partial t} = \nabla\cdot\left(\nabla U(\bm{x})\rho\right) + \beta^{-1}\Laplacian\rho, \quad \rho(\bm{x},0)=\rho_0(\bm{x}).
\end{equation}
To this end we first introduce the Wasserstein metric, the free energy, and the Kullback-Leibler divergence.

The \textbf{Wasserstein-2 distance} $W_{2}\left(\rho_1,\rho_2\right)$ between a pair of PDFs $\rho_1(\bm{x}),\rho_2(\bm{y})\in\cD$ (or, even between probability measures, in general), supported on $\mathcal{X}, \mathcal{Y} \subseteq \mathbb{R}^{n}$, is
\begin{align}
\!\!\!\!W_{2}\left(\rho_1,\rho_2\right) := \left(\underset{\mathrm{d}\sigma\in\Pi(\rho_1,\rho_2)}{\inf} \: \!\int_{\mathcal{X} \times \mathcal{Y}}\: \!\!\!\!\!\!\!\!\|\bm{x} - \bm{y}\|_{2}^{2} \:\mathrm{d}\sigma(\bm{x},\bm{y})\!\right)^{\frac{1}{2}}
\label{WassDefn}	
\end{align}
where $\Pi\left(\rho_1, \rho_2\right)$ is a probability measure on the product space $\cX\times \cY$ having finite second moments and marginals $\rho_1,\rho_2$, respectively.
It is well known that $W_{2} : \cD \times \cD \mapsto [0,\infty)$ is a metric \cite[p.\ 208]{VillaniBook2003}. Further, its square $W_2^2(\rho_1,\rho_2)$ represents the smallest amount of ``work'' needed to ``morph'' $\rho_1$ into $\rho_2$ \cite{French}. The infimum is achieved over a space of measures, and under mild assumptions, the minimizing $\mathrm{d}\sigma$ has support on the graph of the optimal ``transportation map'' $T:\cX\mapsto\cY$ that pushes $\rho_1$ to $\rho_2$. Alternatively, one may view the optimization problem in  \eqref{WassDefn} as seeking the joint distribution of two random vectors $\bbx$ and $\bby$, distributed according to $\rho_1$ and $\rho_2$ respectively, that minimizes the variance $\mathbb{E}\left\{\|\bbx-\bby\|_{2}^2\right\}$.

Another important notion of distance that enters into our discussion, which however is not a metric, is the {\bf Kullback-Leibler divergence} (also known as {\bf relative entropy}) between PDFs or positive measures in general. This is given by
$
D_{\mathrm{KL}}\left(\mathrm{d}\rho_1\|\mathrm{d}\rho_2\right):=\int (\frac{\mathrm{d}\rho_1}{\mathrm{d}\rho_2})\log(\frac{\mathrm{d}\rho_1}{\mathrm{d}\rho_2})\mathrm{d}\rho_2
$
where $\frac{\mathrm{d}\rho_1}{\mathrm{d}\rho_2}$ denotes the Radon-Nikodym derivative. When\footnote{Here we use a slight abuse of notation in that we denote both, the measure and the density with the same symbol.}$\mathrm{d}\rho_i=\rho_i(\bbx)\mathrm{d}\bbx$, $i\in\{1,2\}$, are absolutely continuous with respect to the Lebesgue measure, then \[D_{\mathrm{KL}}\left(\mathrm{d}\rho_1\|\mathrm{d}\rho_2\right)=\int_{\mathbb{R}^{n}} \rho_1(\bm{x})\log\frac{\rho_1(\bm{x})}{\rho_2(\bm{x})}\mathrm{d}\bbx.\]

Gradient flow requires an {\bf energy functional}, which we denote by
$\mathcal{E}\left(\rho\right):=\int U(\bm{x})\rho(\bm{x})\mathrm{d}\bbx$, where $U(\cdot)$ is the potential energy. Then, a stochastically driven gradient flow is modeled by the It\^{o} SDE
\eqref{jkoSDE}	
and the Fokker-Planck equation
\eqref{eq:FP}
for the corresponding PDF as before. The stationary solution of \eqref{eq:FP} is the Gibbs distribution $\rho_\infty(\bm{x})=\frac{1}{Z}e^{-\beta U(\bm{x})}$, where the normalization constant $Z:=\int_{\mathbb{R}^{n}} e^{-\beta U(\bbx)}\mathrm{d}\bbx$ is known as the \textbf{partition function}. The distance to equilibrium which, in a way, quantifies the amount of work that the system can deliver, is captured by the so-called \textbf{free energy functional} $\mathcal{F}\left(\rho\right)$, defined as the sum of the energy functional $\mathcal{E}\left(\rho\right)$ and the negative differential entropy
$\mathcal{S}\left(\rho\right)$ given in \eqref{eq:entropy}, that is,
\begin{subequations}\label{FreeEnergyDefn}
\begin{align}
	\mathcal{F}\left(\rho\right) &:= \mathcal{E}\left(\rho\right) \: + \beta^{-1} \: \mathcal{S}\left(\rho\right)
	\label{FreeEnergyDefn_a}\\
	&\phantom{:}=\beta^{-1}D_{\mathrm{KL}}\left(\rho\| e^{-\beta U(\bm{x})}\right).\label{FreeEnergyDefn_b}
\end{align}
\end{subequations}

%Given $\rho\in\cD$, the \textbf{free energy functional} $\mathcal{F}\left(\rho\right)$ is defined as a sum of the energy functional $\mathcal{E}\left(\rho\right):=\int U(\bm{x})\rho(\bm{x})d\bm{x}$, and the scaled negative of the differential entropy $\mathcal{S}\left(\rho\right):=-(-\int\rho(\bm{x})\log\rho(\bm{x})d\bm{x})$ wherein the term inside the last parenthesis is the differential entropy. In other words, 
%
%where $U(\bm{x}):\mathbb{R}^{d}\mapsto [0,\infty)$ is the potential energy corresponding to state $\bm{x}\in\mathbb{R}^{d}$.
%
%\begin{remark}
%\label{FreeEnergywithKLremark}
%Recalling that for two given PDFs $p(\bm{x})$ and $q(\bm{x})$, the KL-divergence , from (\ref{FreeEnergyDefn}) it is then easy to verify that
%\begin{eqnarray}
%\mathcal{F}\left(\rho\right) = \beta^{-1}D_{\mathrm{KL}}\left(\rho \parallel e^{-\beta U(\bm{x})}\right).
%\label{FreeEnergyAsKLdivergence}	
%\end{eqnarray}	
%\end{remark}
%
%
%\begin{remark}
%\label{ChoiceofbetaUremark}
%The choice of the parameter $\beta$, and the potential energy $U(\bm{x})$ in (\ref{FreeEnergyDefn}) are governed by the drift and diffusion coefficients of the SDE for state $\bm{x}(t)$. In \cite{JKO1998}, 

For the case of \eqref{jkoSDE}, the JKO scheme becomes
\begin{eqnarray}
\varrho_{k}\left(\bm{x},h\right) = \underset{\varrho\in\cD_{2}}\arginf  \{ \displaystyle\frac{1}{2}W_{2}^{2}\left(\varrho, \varrho_{k-1}\right) \: + \: h \: \mathcal{F}\left(\varrho\right)\}, \: k\in\mathbb{N},
\label{JKO}	
\end{eqnarray}
for step-size $h>0$, and initialized by a given $\varrho_{0}$ (satisfying $\mathcal{F}(\varrho_{0}) < \infty$). For $U(\bbx)\equiv 0$, \eqref{JKO} reduces to \eqref{eq:JKO}.
Solving (\ref{JKO}) results in a sequence of PDFs $\{\varrho_{k}(\bm{x},h)\}_{k\in\mathbb{N}}$ in $\cD_{2}$. 
It can be shown following \cite{JKO1998} that
 $\varrho_{k}(\bm{x},h)\rightharpoonup \rho\left(\bm{x}(t),t\right)$ weakly in $L^{1}(\mathbb{R}^{n})$ for $t\in[(k-1)h, kh)$, $k\in\mathbb{N}$, as $h\downarrow 0$.

%The Fokker-Planck PDE corresponding to (\ref{jkoSDE}) reads $\frac{\partial\rho}{\partial t} = \nabla\cdot\left(\nabla U(\bm{x})\rho\right) + \beta^{-1}\Laplacian\rho$, whose stationary solution yields the steady-state PDF $\rho_{\infty}(\bm{x}) \propto e^{-\bet\input{JKO_AH_v0.3a.tex}
%a U(\bm{x})}$ up to a normalization constant, and consequently Definition \ref{FreeenergyDefnLabel} is natural in view of Remark \ref{FreeEnergywithKLremark}. In other words, if one can express the It\^{o} SDE governing the state dynamics in the form (\ref{jkoSDE}), then the $\beta$ and $U(\bm{x})$ to be used in (\ref{FreeEnergyDefn}) or (\ref{FreeEnergyAsKLdivergence}) can simply be ``read off" from the SDE. 

%In this paper, we will focus on a linear stochastic system given by (\ref{OUmultivariate}) and show how to write $\mathcal{F}(\varrho)$ for the same. We will use the resulting $\mathcal{F}(\varrho)$ to compute $\{\varrho_{k}(\bm{x},h)\}_{k\in\mathbb{N}}$ via the JKO scheme, and then establish the consistency of our solution with that of the Fokker-Planck solution $\rho(\bbx(t),t)$ by explicit calculations.   
%\end{remark}

%%%%%%%%%%%%%%%%%%%%%%%%%%%%%%%%%%%%%%%%%%%%%%%%%%%%%%%%%%%%%%%%%%%%%%%%%%%%%%%
%%%%%%%%%%%%%%%%%%%%%%%%%%%%%%%%%%%%%%%%%%%%%%%%%%%%%%%%%%%%%%%%%%%%%%%%%%%%%%%
%%%%%%%%%%%%%%%%%%%%%%%%%%%%%%%%%%%%%%%%%%%%%%%%%%%%%%%%%%%%%%%%%%%%%%%%%%%%%%%
\section{JKO Scheme for linear Gaussian systems}\label{UncPropSectionLabel}

We now develop and solve the JKO scheme for the linear Gaussian system in (\ref{OUmultivariate}) with $\rho_{0} = \mathcal{N}\left(\bm{\mu}_{0},\bm{P}_{0}\right)$ and $\bm{Q}\equiv \bm{I}$, without loss of generality. Therefore, we are concerned with the linear Fokker-Planck (Kolmogorov's forward) PDE
\begin{eqnarray}
\displaystyle\frac{\partial\rho}{\partial t} = - \nabla \cdot \left(\rho\bm{A}\bm{x}\right) \: + \: \displaystyle\frac{1}{2}\displaystyle\sum_{i,j=1}^{n}%\sum_{j=1}^{d}
\displaystyle\frac{\partial^{2}}{\partial x_{i}\partial x_{j}}\left(\rho\bm{B}\bm{Q}\bm{B}^{\top}\right)_{ij}.
\label{FokkerPlanckLinear}	
\end{eqnarray}
%where $\bbx\in\mathbb{R}^{n}$, and $\rho(\bm{x}(0), 0) = \rho_{0}(\bm{x})$. This governs the transient evolution of the state density $\rho(\bbx(t),t)$ for the linear stochastic system
%\begin{eqnarray}
%\mathrm{d}\bm{x}(t) = \bm{A}\bm{x}(t) \: \mathrm{d}t\: + \bm{B}\:\mathrm{d}\bm{w}(t), \qquad \bm{x}(0) \sim \rho_{0}(\bm{x}),
%\label{OUmultivariate}
%\end{eqnarray}
%where $\bm{w}(t) \in \mathbb{R}^{q}$ is a Wiener process satisfying $\mathbb{E}\left[\mathrm{d}w_{i}\mathrm{d}w_{j}\right] = \bm{Q}_{ij}\mathrm{d}t \:\forall\:i,j=1,\hdots,q$, with $\bm{Q} \succ \bm{0}$. We suppose that the uncertain initial condition $\bm{x}(0)$ has a known PDF $\rho_{0}(\bm{x})$, the matrix $\bm{A}$ is Hurwitz, and that the diffusion matrix $\bm{B}$ is such that $(\bm{A},\bm{B})$ is a controllable pair.

Under the stated assumptions, it is well-known that (\ref{FokkerPlanckLinear}) admits a steady-state, which is Gaussian with mean zero and covariance $\bm{P}_{\infty} \succ \bm{0}$ that uniquely solves the algebraic Lyapunov equation $\bm{A}\bm{P}_{\infty} + \bm{P}_{\infty} \bm{A}^{\top} + \bm{BQB^{\top}} = \bm{0}$. Also, starting from $\rho_{0}(\bm{x})=\mathcal{N}\left(\bm{\mu}_0,\bm{P}_0\right)$, the transient is $\rho(\bm{x}(t),t)=\mathcal{N}\left(\bm{\mu}(t),\bm{P}(t)\right)$ where the $\bbmu(t)$ and $\bbP(t)$ satisfy the following ordinary differential equations (ODEs) \cite[Ch. 3.6]{AstromBook1970}
\begin{subequations}\label{MeanCovODEs}
\begin{align}\label{MeanCovODEsa}
&\dot{\bm{\mu}}(t) = \bm{A}\bm{\mu}(t), \: \bm{\mu}(0) = \bm{\mu}_{0},\\\label{MeanCovODEsb}
&\dot{\bm{P}}(t) = 	\bm{A}\bm{P}(t) + \bm{P}(t) \bm{A}^{\top} + \bm{BQB^{\top}}, \: \bm{P}(0) = \bm{P}_{0}.
\end{align}
\end{subequations}

Below, we recover these equations using the JKO scheme. First, in Section \ref{SpecialCaseSubsectonLabel}, we explain how this is done when $\bbA$ is symmetric and $\bm{B}\equiv \sqrt{2\beta^{-1}}\bm{I}$, $\beta>0$, in which case, $\bbA \bbx=-\nabla U(\bbx)$ for a suitable potential.
The general case, in Section \ref{GenericCaseSubsectonLabel}, is more involved and requires to view the drift as the gradient of a time-varying potential.

%%%%%%%%%%%%%%%%%%%%%%%%%%%%%%%%%%%%%%%%%%%%%%%%%%%%%%%%%%%%%%%%%%%%%%%%%%%%%%%

%\subsection{Preparatory Lemmas}

%%%%%%%%%%%%%%%%%%%%%%%%%%%%%%%%%%%%%%%%%%%%%%%%%%%%%%%%%%%%%%%%%%%%%%%%%%%%%%%

\subsection{The case where $\bbA$ is symmetric and $\bm{B}\equiv \sqrt{2\beta^{-1}}\bm{I}$}\label{SpecialCaseSubsectonLabel}

Since $\bm{B}\equiv \sqrt{2\beta^{-1}}\bm{I}$, the pair $(\bm{A},\bm{B})$ is controllable. Further, since $\bm{A}$ is Hurwitz and symmetric, $\bm{\Gamma}:=-\bm{A} \succ 0$, and utilizing the potential
\[
U(\bm{x}):=\frac{1}{2}\bm{x}^{\top}\bm{\Gamma}\bm{x} \geq 0,
\]
we can cast (\ref{OUmultivariate}) in the canonical form (\ref{jkoSDE}). Then,
\[
\mathcal{E}(\varrho) := \mathbb{E}[U(\bm{x})] = \frac{1}{2}\left(\bm{\mu}^{\top}\bm{\Gamma}\bm{\mu} + \tr\left(\bm{\Gamma}\bm{P}\right)\right),
\]
where $\bm{P}$ is the covariance of $\bbx$. Notice that $\mathcal{E}(\cdot)$ depends on the PDF of $\bbx$ {\em only via its mean and covariance}.

To carry out the optimization (\ref{JKO}) over $\cD_{2}$, we adopt a \textbf{two-step strategy}. Our approach is motivated by the observation that the objective function in (\ref{JKO}) is a sum of two functionals. In the \textbf{first step}, we choose a {\em suitable} parameterized subset of $\cD_{2}$ in such a way that when we optimize the functionals $\frac{1}{2}W^{2}(\varrho,\varrho_{0})$ and $h\mathcal{F}(\varrho)$ {\em individually over this chosen subspace}, the arginfs (which are achieved) of the two individual optimization problems match.
Hence, the sum of the two has the same arginf over the chosen subspace. In the \textbf{second step}, we optimize over the subspace parameters.
Our choice for the parameterized set of densities is $\cD_{\bm{\mu},\bm{P}} \subset \cD_{2}$, i.e., the PDFs with given mean-covariance pair $(\bm{\mu},\bm{P})$; the choice of the optimal pair is to be decided in the second optimization step.

The development below requires several technical lemmas that are collected in the Appendix.

\subsubsection{Optimizing over $\cD_{\bm{\mu},\bm{P}}$}\label{FirstStepSpecialCase}
Given $\varrho_{0}\equiv\rho_{0}=\mathcal{N}(\bm{\mu}_{0},\bm{P}_{0})$, and a $\bm{\mu}$ and $\bm{P}\succ \bm{0}$, we first determine
\begin{eqnarray}
\varrho_{1} = \underset{\varrho\in\cD_{\bm{\mu},\bm{P}}}\arginf \{ \displaystyle\frac{1}{2}W_{2}^{2}\left(\varrho, \mathcal{N}(\bm{\mu}_{0},\bm{P}_{0})\right) \: + \: h \: \mathcal{F}\left(\varrho\right)\}.
\label{SubspaceOptimization}	
\end{eqnarray}
From Lemma \ref{WsqareProjLemma} we see that $\underset{\varrho\in\cD_{\bm{\mu},\bm{P}}}\arginf \displaystyle\frac{1}{2}W_{2}^{2}\left(\varrho, \mathcal{N}(\bm{\mu}_{0},\bm{P}_{0})\right)$ is achieved by $\varrho=\mathcal{N}(\bm{\mu},\bm{P})$ (uniquely). From Lemma \ref{infFreeEnergyLemma}, since $U(\bm{x}) = \frac{1}{2}\bm{x}^{\top}\bm{\Gamma}\bm{x}$, we also know that $\underset{\varrho\in\cD_{\bm{\mu},\bm{P}}}\arginf h\:\mathcal{F}\left(\varrho\right)$ is achieved by $\varrho=\mathcal{N}(\bm{\mu},\bm{P})$ (uniquely). Thus,  $\varrho_{1} = \mathcal{N}(\bm{\mu},\bm{P})$. The infimal value in (\ref{SubspaceOptimization}) is now the sum of the two infima,
{\small{\begin{align}
&\displaystyle\frac{1}{2}\left[\parallel\bm{\mu}-\bm{\mu}_{0}\parallel_{2}^{2} + \tr\left(\bm{P} + \bm{P}_{0} - 2 \left(\bm{P}_{0}^{\frac{1}{2}} \bm{P} \bm{P}_{0}^{\frac{1}{2}}\right)^{\frac{1}{2}}\right)\right] + \displaystyle\frac{h}{2\beta}\nonumber\\
&\left[-n - n\log(2\pi) -\log\det(\bm{P})+ \beta\bm{\mu}^{\top}\bm{\Gamma}\bm{\mu} + \beta\tr\left(\bm{\Gamma}\bm{P}\right)\right].
\label{InfValue}	
\end{align}}}

\subsubsection{Optimizing over $(\bm{\mu},\bm{P})$}
Equating the gradient of (\ref{InfValue}) w.r.t. $\bm{\mu}$ to zero, results $\bm{\mu} = \bm{\phi}(\bm{\mu_{0}}) := (\bm{I}+h\bm{\Gamma})^{-1}\bm{\mu}_{0}$. The recursion $\bm{\mu}_{k} = \bm{\phi}(\bm{\mu_{k-1}})$, up to first order in $h$, becomes
\begin{eqnarray}
\bm{\mu}_{k} = \left(\bm{I} \:-\: h\bm{\Gamma}\right)\bm{\mu}_{k-1} \: + \: O(h^{2}).
\label{SplCaseRecoverMu}	
\end{eqnarray}
We see that this recursion coincides with the solution of  (\ref{MeanCovODEsa}) in the ``small $h$'' limit. Specifically, 
 $\bm{\mu}(t)=e^{\bm{A}t}\bm{\mu}_{0} \Rightarrow \bm{\mu}_{k} := \bm{\mu}(t=kh) = \left(e^{\bm{A}h}\right)^{k}\bm{\mu}_{0} \Rightarrow \bm{\mu}_{k} = e^{\bm{A}h} \bm{\mu}_{k-1} = \left(\bm{I} + h\bm{A}\right)\bm{\mu}_{k-1} + O(h^{2})$, which is same as (\ref{SplCaseRecoverMu}) since $\bm{\Gamma} := -\bm{A}$. Thus, we have recovered (\ref{MeanCovODEsa}) using discrete time-stepping via JKO scheme in the small step-size limit.

Setting the gradient of (\ref{InfValue}) w.r.t. $\bm{P}$ to zero (using Lemma \ref{MatrixDerivativeofWassCrossTermLemma}), we obtain
\begin{eqnarray}
\bm{I} - \bm{P}_{0}^{\frac{1}{2}}\left(\bm{P}_{0}^{-\frac{1}{2}} \bm{P}^{-1} \bm{P}_{0}^{-\frac{1}{2}}\right)^{\frac{1}{2}}\bm{P}_{0}^{\frac{1}{2}} - \frac{h}{\beta}\bm{P}^{-1} + h\bm{\Gamma} = \bm{0}.
\label{partialPequalsZero}	
\end{eqnarray}
By pre and post multiplying both sides of (\ref{partialPequalsZero}) with $\bm{P}_{0}^{-\frac{1}{2}}$, and letting $\left(\bm{P}_{0}^{-\frac{1}{2}} \bm{P}^{-1} \bm{P}_{0}^{-\frac{1}{2}}\right)^{\frac{1}{2}}=:\bm{Z}$, we arrive at
\[
\bm{Z}^{2} + \frac{\beta}{h}\bm{Z} - \frac{\beta}{h}\bm{P}_{0}^{-\frac{1}{2}}(\bm{I}+h\bm{\Gamma})\bm{P}_{0}^{-\frac{1}{2}} = \bm{0},
\]
which admits the unique closed-form solution \cite[p. 304]{HighamKimSIAM2001}
\begin{eqnarray}
\bm{Z} = \frac{\beta}{2h}\left(-\bm{I} \: + \: \left(\bm{I} + 4\frac{h}{\beta}\bm{P}_{0}^{-\frac{1}{2}}\left(\bm{I}+h\bm{\Gamma}\right)\bm{P}_{0}^{-\frac{1}{2}}\right)^{\frac{1}{2}}\right).
\label{SolutionQuadraticMatrixEquation}	
\end{eqnarray}
Expanding (\ref{SolutionQuadraticMatrixEquation}), we obtain
{\small{\begin{flalign}
&\bm{Z} = \frac{\beta}{2h}\left[-\bm{I} + \bigg\{\bm{I} + \frac{1}{2}4\frac{h}{\beta}\bm{P}_{0}^{-\frac{1}{2}}\left(\bm{I}+h\bm{\Gamma}\right)\bm{P}_{0}^{-\frac{1}{2}} + \right.\nonumber\\
&\left.\displaystyle\frac{\frac{1}{2}\left(\frac{1}{2}-1\right)}{2!}\frac{16h^{2}}{\beta^{2}}\bm{P}_{0}^{-\frac{1}{2}}\left(\bm{I}+h\bm{\Gamma}\right)\bm{P}_{0}^{-1}\left(\bm{I}+h\bm{\Gamma}\right)\bm{P}_{0}^{-\frac{1}{2}} + O(h^{3})\bigg\}\right]\nonumber\\
&= \bm{P}_{0}^{-\frac{1}{2}} \left(\bm{I} + h\bm{\Gamma} - \frac{h}{\beta}\bm{P}_{0}^{-1}\right) \bm{P}_{0}^{-\frac{1}{2}} \: + \: O(h^{2}).
\label{MatrixSeriesExpansion}	
\end{flalign}}}
\hspace*{-0.05in}Substituting $\bm{Z} = \left(\bm{P}_{0}^{-\frac{1}{2}} \bm{P}^{-1} \bm{P}_{0}^{-\frac{1}{2}}\right)^{\frac{1}{2}}$ back into (\ref{MatrixSeriesExpansion}), squaring, and rearranging, we get that
{\footnotesize\begin{align*}
&\bm{P} = \!\left(\bm{I}+h\left(\bm{\Gamma}-\frac{1}{\beta}\bm{P}_{0}^{-1}\right)\right)^{-1}\!\!\!\!\bm{P}_{0}\left(\bm{I}+h\left(\bm{\Gamma}-\frac{1}{\beta}\bm{P}_{0}^{-1}\right)\right)^{-1}\!\!\!+\! O(h^{2})\\
&= \left(\bm{I}-h\left(\bm{\Gamma}-\frac{1}{\beta}\bm{P}_{0}^{-1}\right)\right)\bm{P}_{0} \left(\bm{I}-h\left(\bm{\Gamma}-\frac{1}{\beta}\bm{P}_{0}^{-1}\right)\right) \: + \: O(h^{2})\\
& = \bm{\Psi}(\bm{P}_{0})\allowbreak \: + \: O(h^{2}),
\end{align*}}
\hspace*{-.05in}where $\bm{\Psi}(\bm{P}_{0}):=\bm{P}_{0} \: + \: h\left(-\bm{\Gamma}\bm{P}_{0} - \bm{P}_{0}\bm{\Gamma} + \allowbreak 2\allowbreak\beta^{-1}\allowbreak\bm{I}\allowbreak\right)$.  Set the matrix-valued recursion $\bm{P}_{k}=\bm{\Psi}\left(\bm{P}_{k-1}\right)$, where
{\small{\begin{eqnarray}
\bm{\Psi}\left(\bm{P}_{k-1}\right) := \bm{P}_{k-1} + h\left(-\bm{\Gamma}\bm{P}_{0} - \bm{P}_{0}\bm{\Gamma} + 2\beta^{-1}\bm{I}\right) + O(h^{2}).
\label{CovarianceRecursion}
\end{eqnarray}}}
\hspace*{-.05in}To show that (\ref{CovarianceRecursion}) indeed recovers \eqref{MeanCovODEsb}, first notice that substituting $\bm{A} = \bm{A}^{\top} = -\bm{\Gamma}$ and $\bm{B}=\sqrt{2\beta^{-1}}\bm{I}$ in (\ref{MeanCovODEsb}) results the Lyapunov differential equation
\[
\dot{\bm{P}}(t)=-\bm{\Gamma}\bm{P}(t)-\bm{P}(t)\bm{\Gamma}+2\beta^{-1}\bm{I}
\]
subject to $\bm{P}(0)=\bm{P}_{0}$, which can be solved via the method of integrating factor as
\begin{eqnarray}
\bm{P}(t) = \frac{1}{\beta}\bm{\Gamma}^{-1}\left(\bm{I} - e^{-2\bm{\Gamma}t}\right) \: + \: e^{-\bm{\Gamma}t}\bm{P}_{0}e^{-\bm{\Gamma}t}.
\label{SpecialCaseLyapODEsoln}	
\end{eqnarray}
Thus, for $t=kh$, (\ref{SpecialCaseLyapODEsoln}) gives
\begin{align*}
\bm{P}_{k}&:=\bm{P}(kh)=\beta^{-1}\bm{\Gamma}^{-1}\left(\bm{I}-e^{-2\bm{\Gamma}kh}\right) + e^{-\bm{\Gamma}kh}\bm{P}_{0}e^{-\bm{\Gamma}kh}\\
%& \hspace*{-15pt}= \beta^{-1}\bm{\Gamma}^{-1}\left(\bm{I}-\left(\bm{I}-2kh\bm{\Gamma}\right)\right) + (\bm{I}-kh\bm{\Gamma})\bm{P}_{0} (\bm{I}-kh\bm{\Gamma}) + O(h^{2})\\
&= 2\beta^{-1}kh\bm{I} + \left(\bm{P}_{0} - kh\bm{\Gamma}\bm{P}_{0} - kh\bm{P}_{0}\bm{\Gamma}\right) + O(h^{2}).
\end{align*}
Replacing $k$ with $k-1$ in the latter yields a similar expression for $\bm{P}_{k-1}$. Then, subtracting these expressions for $\bm{P}_{k-1}$ from $\bm{P}_{k}$ we obtain that 
\[
\bm{P}_{k}-\bm{P}_{k-1}=2\beta^{-1}h\bm{I} - h\bm{\Gamma}\bm{P}_{0} - h\bm{P}_{0}\bm{\Gamma} + O(h^{2}),
\]
 which is same as (\ref{CovarianceRecursion}) derived from JKO scheme. Thus, we have recovered the covariance evolution through Fokker-Planck dynamics using the time-stepping procedure via JKO scheme in the small step-size limit.

%%%%%%%%%%%%%%%%%%%%%%%%%%%%%%%%%%%%%%%%%%%%%%%%%%%%%%%%%%%%%%%%%%%%%%%%%%%%%%%

\subsection{The case of Hurwitz $\bbA$ and controllable $(\bbA,\bbB)$}\label{GenericCaseSubsectonLabel}
We scale $\bbB$ into $\sqrt{2}\bbB$ without loss of generality, and  take as initial PDF $\varrho_{0}\equiv\rho_{0} = \mathcal{N}(\bbmu_{0},\bbP_{0})$. Since we allow any Hurwitz (not necessarily symmetric) $\bbA$, and any $\bbB$ that makes $(\bbA,\sqrt{2}\bbB)$ a controllable pair, it is not apparent if and how one can express (\ref{OUmultivariate}) in the canonical form (\ref{jkoSDE}). The main impediment in doing so, is twofold: (1) how to define the potential energy $U(\bm{x})$, and (2) how to interpret and define the parameter $\beta$ in the generic case. In the following, we show that by two successive co-ordinate transformations, system (\ref{OUmultivariate}) can indeed be put in the form (\ref{jkoSDE}). Similar transformations have been mentioned in \cite[p. 1464]{LiberzonBrockettSIAM2000}, \cite{BrockettWillemsCDC1978} in a different context.

\subsubsection{Equipartition of energy coordinate transformation}\label{EquipartitionSubsubsection}
Consider the stationary covariance $\bbP_{\infty}$ associated with $(\bbA,\sqrt{2}\bbB)$ that satisfies 
\begin{equation}\label{eq:lyapunov}
\bbA\bbP_{\infty} + \bbP_{\infty}\bbA^{\top} + 2\bbB\bbB^{\top} = \bm{0}.
\end{equation}
For a system at a stationary distribution, we define the \textbf{thermodynamic temperature} $\theta$ as the average amount of ``energy'' per degree of freedom, that is,
\[
\theta := \frac{1}{n}\tr(\bbP_{\infty}),
\]
and, thereby, $\beta:=\theta^{-1}$ the \textbf{inverse temperature}. By
 pre and post multiplying \eqref{eq:lyapunov} with $\bbP_{\infty}^{-\frac{1}{2}}$, and rescaling by $\theta$ so as to preserve the temperature in the new coordinates, we get 
\begin{eqnarray}
\bbA\ep\theta\bbI + \theta\bbI\bbA\ep^{\top} + \sqrt{2\theta}\bbB\ep(\sqrt{2\theta}\bbB\ep)^{\top}=\bm{0},
\label{LyapunovInY} 
\end{eqnarray}
where $\bbA\ep := \bbP_{\infty}^{-\frac{1}{2}}\bbA\bbP_{\infty}^{\frac{1}{2}}$, $\bbB\ep := \bbP_{\infty}^{-\frac{1}{2}}\bbB$, while the stationary covariance $\theta\bbI$ reflects {\em equipartition of energy}. The equipartition of energy co-ordinate transformation $(\bbA,\sqrt{2}\bbB)\mapsto (\bbA\ep,\sqrt{2\theta}\bbB\ep)$, corresponds to the state-transformation $\bbx \mapsto \bbx\ep := \sqrt{\theta}\bbP_{\infty}^{-\frac{1}{2}} \bbx$, leading to
\begin{eqnarray}
\mathrm{d}\bbx\ep(t) = \bbA\ep\bbx\ep(t) \:\mathrm{d}t \: + \: \sqrt{2\theta}\bbB\ep\:\mathrm{d}\bbw(t).
\label{ySDE}	
\end{eqnarray}
\noindent This settles how $\beta$ is to be defined and interpreted in the context of JKO scheme  \eqref{FreeEnergyDefn} and (\ref{JKO}). On the other hand, $\bbA\ep$ being similar to $\bbA$, is guaranteed to be Hurwitz but not symmetric, unless $\bbA$ was symmetric to begin with. Thus, it remains for us to ``symmetrize" $\bbA\ep$ and define a suitable potential energy $U(\cdot)$ as needed in (\ref{JKO}). We do this next.

\subsubsection{Symmetrization transformation}\label{SymmetrizationSubsubsection}
We introduce the time-varying transformation
\[
\bbx\ep \mapsto \bbx_{\rm{sym}} := e^{-\bbA\ep\skewsym t} \bbx\ep
\]
where $\bbA\ep\skewsym := \frac{1}{2}(\bbA\ep - \bbA\ep^{\top})$. This results in
\[
(\bbA\ep,\sqrt{2\theta}\bbB\ep) \mapsto (\bbF(t),\sqrt{2\theta}\bbG(t)),
\]
with 
\[
\bbF(t) := e^{-\bbA\ep\skewsym t} \bbA\ep\sym e^{\bbA\ep\skewsym t},\mbox{ and }\bbG(t) := e^{-\bbA\ep\skewsym t} \bbB\ep,
\]
where, similarly, $\bbA\ep\sym := \frac{1}{2}(\bbA\ep + \bbA\ep^{\top})$.
Thus, $\bbx_{\rm{sym}}(t)$ satisfies
\begin{eqnarray}
\mathrm{d}\bbx_{\rm{sym}}(t) = \bbF(t)\bbx_{\rm{sym}}(t) \:\mathrm{d}t \: + \: \sqrt{2\theta}\bbG(t)\:\mathrm{d}\bbw(t).
\label{zSDE}	
\end{eqnarray}

Notice that $\bbF(t)$ is symmetric for all $t$. Furthermore, observe that the new coordinates $\bbx_{\rm{sym}}$ is simply obtained by a (time-varying) orthogonal transformation of the equipartition of energy coordinates $\bbx\ep$. Hence the stationary covariance of $\bbx_{\rm{sym}}$ is identical to that of $\bbx\ep$, which is $\theta\bbI$ (from Section \ref{EquipartitionSubsubsection}). What happens is that the covariance of $\bbx_{\rm{sym}}(t)$ tends to the same steady state value as $t\to\infty$ {\em in spite of the fact that \eqref{zSDE} has time varying coefficients}.
To see this in different way, we can rewrite (\ref{LyapunovInY}) as $\bbB\ep\bbB\ep^{\top} = - \bbA\ep\sym$, and deduce that
\begin{eqnarray}
\bbG(t)\bbG(t)^{\top} = e^{-\bbA\ep\skewsym t} \bbB\ep\bbB\ep^{\top} e^{\bbA\ep\skewsym t} = - \bbF(t),\nonumber\\
\Rightarrow \bbF(t)\theta\bbI + \theta\bbI\bbF(t) + \sqrt{2\theta}\bbG(t)(\sqrt{2\theta}\bbG(t))^{\top}=\bm{0}.
\label{DeducingLyapunovInZ}	
\end{eqnarray}
The symmetrization $\bbx\ep \mapsto \bbx_{\rm{sym}}$ leaves the stationary covariance $\theta\bbI$ invariant. This guarantees that the definition of temperature $\theta$ stays intact. The coordinate transformations described above are summarized in Table \ref{TransformationTable}.

%Given some matrix $\bbY$, defining the operator $\nabla_{\bbY\bbY^{\top}}$ (see for example, \cite[p. 1463]{LiberzonBrockettSIAM2000}, \cite{Brockett1997notes}) with respect to the metric $(\bbY\bbY^{\top})^{-1}$ on $\bm{z}\in\mathbb{R}^{d}$ as $\bbY\bbY^{\top} \displaystyle\frac{\partial}{\partial\bbz}$, and taking $U(\bbz):=\frac{1}{4}\bbz^{\top}(\theta\bbI)^{-1}\bbz$, we get
%\begin{eqnarray}
%- \nabla_{2\theta\bbG(t)\bbG(t)^{\top}}U(\bbz) = -2\theta\bbG(t)\bbG(t)^{\top}\:\frac{1}{2\theta}\bbz = \bbF(t)\bbz,
%\label{FzAsGradient}	
%\end{eqnarray}
%where the last equality follows from the algebraic Lyapunov equation in the $\bbz$ coordinates, given by $2\theta\bbF(t) + 2\theta\bbG(t)\bbG(t)^{\top} = \bm{0}$. This allows us to write the SDE (\ref{zSDE}) in the form (\ref{jkoSDE}), that is,
%\begin{eqnarray}
%\mathrm{d}\bbz(t) = - \nabla_{2\theta\bbG(t)\bbG(t)^{\top}}U(\bbz) \:\mathrm{d}t \: + \: \sqrt{2\theta}\bbG(t)\:\mathrm{d}\bbw(t).
%\label{zSDEinCanonicalForm}	
%\end{eqnarray}

\begin{table*}[!htb]
\centering
\begin{tabular}{| c | c | c | c |}
\hline
\backslashbox{Attribute $\downarrow$}{Coordinate $\rightarrow$} & Original & Equipartition of energy & Symmetrization\\ 
\hline\hline
& & &\\
State vector & $\bbx$ & $\bbx\ep$ & $\bbx_{\rm{sym}}$ \\
& & &\\
\hline
& & &\\
System matrices & $(\bbA,\sqrt{2}\bbB)$ & $(\bbA\ep,\sqrt{2\theta}\bbB\ep)$ & $(\bbF(t),\sqrt{2\theta}\bbG(t))$\\
& & &\\
\hline
& & &\\
Stationary covariance & $\bbP_{\infty}$ & $\theta\bbI$ & $\theta\bbI$\\
& & &\\
\hline
\end{tabular}
\caption{Summary of the coordinate transformations for Section \ref{GenericCaseSubsectonLabel}.}
\label{TransformationTable}
\end{table*}

\subsubsection{Recovery of the Fokker-Planck solution}\label{GenericLinFPKrecov}
We are now ready to apply the JKO scheme to the generic stochastic linear system $\mathrm{d}\bm{x}(t) = \bm{A}\bm{x}(t) \: \mathrm{d}t\: + \sqrt{2}\bm{B}\:\mathrm{d}\bm{w}(t)$, with initial PDF $\rho(\bbx(0),0) = \mathcal{N}(\bm{\mu}_{0},\bbP_{0})$. To this end, we carry out a computation akin to the two steps in Section \ref{SpecialCaseSubsectonLabel}, for the transformed SDE (\ref{zSDE}) in the symmetrized coordinate $\bbx_{\rm{sym}}$. From there on, we recover the Fokker-Planck solution in the original coordinate $\bbx$. 

%First, notice that $\bbx_{\rm{sym}} = e^{-\bbA\ep\skewsym t} \sqrt{\theta} \bbP_{\infty}^{-\frac{1}{2}} \bbx$ 
Since $\bbx \mapsto \bbx_{\rm{sym}}$ is a linear transformation, it follows that $\bbx_{\rm{sym}} \sim \mathcal{N}(\bbmu_{\rm{sym}}, \bbP_{\rm{sym}})$ whenever $\bbx\sim \mathcal{N}(\bbmu, \bbP)$. Thus, carrying out the first step of the optimization in $\bbx_{\rm{sym}}$ coordinate, we get an expression similar to (\ref{InfValue}) wherein $(\bbmu,\bbP)$ is to be replaced by $(\bbmu_{\rm{sym}},\bbP_{\rm{sym}})$, and $(\bbmu_{0},\bbP_{0})$ is to be replaced by $(\bbmu_{{\rm{sym}}_{0}},\bbP_{{\rm{sym}}_{0}})$. To carry out the second step of optimization, notice that $\bbA\ep\sym = -\bbB\ep\bbB\ep^{\top} \preceq 0$, and consequently $\bbF(t) = e^{-\bbA\ep\skewsym t} \bbA\ep\sym e^{\bbA\ep\skewsym t} \preceq 0$. Thus, considering the time-varying potential 
\[
U(\bbx_{\rm{sym}}) := -\frac{1}{2}\bbx_{\rm{sym}}^{\top}\bbF(t)\bbx_{\rm{sym}} \geq 0,
\]
and setting the partial derivative of the infimal value from first stage of the optimization w.r.t. $\bbmu_{\rm{sym}}$ to zero, results the recursion $\bbmu_{{\rm{sym}}_{k}} = (\bbI - h\bbF(kh))^{-1} \bbmu_{{\rm{sym}}_{k-1}}$. Recalling that $\bbx_{\rm{sym}} = e^{-\bbA\ep\skewsym t} \sqrt{\theta} \bbP_{\infty}^{-\frac{1}{2}} \bbx$, we arrive at a recursion in original coordinate:
\begin{align}
\bbmu_{k} = \bbP_{\infty}^{\frac{1}{2}} e^{\bbA\ep\skewsym kh} &\{(\bbI - h\bbF(kh))^{-1} e^{\bbA\ep\skewsym h}\} \nonumber\\
&e^{-\bbA\ep\skewsym kh} \bbP_{\infty}^{-\frac{1}{2}} \: \bbmu_{k-1}.
\label{JKOrecursionMeanOriginalCoord}	
\end{align}
By series expansion and collecting linear terms in $h$, one can verify the following: 
\begin{align*}
&(\bbI - h\bbF(kh))^{-1} = \bbI + h\bbA\ep\sym + O(h^{2}),\\
&(\bbI - h\bbF(kh))^{-1} e^{\bbA\ep\skewsym h} = \bbI + h\bbA\ep + O(h^{2}),\\
&e^{\bbA\ep\skewsym kh} (\bbI - h\bbF(kh))^{-1} e^{\bbA\ep\skewsym h} e^{-\bbA\ep\skewsym kh}\\
&\hspace*{1cm}= \bbI + h\bbA\ep + O(h^{2}).
\end{align*}
Hence (\ref{JKOrecursionMeanOriginalCoord}) yields
\begin{align}
\bbmu_{k} &= \left(\bbI + h\bbP_{\infty}^{\frac{1}{2}}\bbA\ep\bbP_{\infty}^{-\frac{1}{2}}\right) \bbmu_{k-1} \: + \: O(h^{2}) \nonumber\\
&= \left(\bbI + h\bbA\right) \bbmu_{k-1} \: + \: O(h^{2}),
\label{JKOmeanrecoveryGeneralcase}
\end{align}
where the last equality follows from $\bbA\ep:=\bbP_{\infty}^{-\frac{1}{2}}\bbA\bbP_{\infty}^{\frac{1}{2}}$. Since $\dot{\bbmu} = \bbA\bbmu$ and $e^{h\bbA} = \bbI + h\bbA + O(h^{2})$, in the small $h$ limit, equation (\ref{JKOmeanrecoveryGeneralcase}) thus recovers \eqref{MeanCovODEsa}, as in Section \ref{SpecialCaseSubsectonLabel}. A similar straightforward but tedious computation leads to the matrix recursion 
\begin{equation}\label{eq:unproven}
\bbP_{k} - \bbP_{k-1} = h(\bbA\bbP_{k-1} + \bbP_{k-1}\bbA^{\top} + 2\bbB\bbB^{\top}) + O(h^{2}),
\end{equation} which in the limit $h\downarrow 0$, is indeed a first-order approximation of the Lyapunov equation for the original system. We omit the details for brevity.

%%%%%%%%%%%%%%%%%%%%%%%%%%%%%%%%%%%%%%%%%%%%%%%%%%%%%%%%%%%%%%%%%%%%%%%%%%%
%%%%%%%%%%%%%%%%%%%%%%%%%%%%%%%%%%%%%%%%%%%%%%%%%%%%%%%%%%%%%%%%%%%%%%%%%%%

\section{JKO-like Schemes for Filtering}
\label{FilteringSectionLabel}
%In the continuous-time filtering problem, a process or state dynamics model like (\ref{jkoSDE}) is appended with an observation model of the form
%\begin{eqnarray}
%\mathrm{d}\bbz(t) = \bbc(\bm{x}(t),t)\:\mathrm{d}t + \mathrm{d}\bbv(t), \qquad \bbz\in\mathbb{R}^{m},
%\label{ObservationSDE}	
%\end{eqnarray}
%where $\bbz(t)$ denotes the sensor measurement at time $t$, and the measurement noise $\bbv(t)$ is a Wiener process satisfying $\mathbb{E}\left[\mathrm{d}v_{i}\mathrm{d}v_{j}\right] = \bbR_{ij}\mathrm{d}t \:\forall\:i,j=1,\hdots,m$, with $\bbR \succ \bm{0}$. As before, the measurement noise $\bbv(t)$ is assumed to be independent of the process noise $\bbw(t)$ in (\ref{jkoSDE}), and also independent of the initial state $\bm{x}(0)$. The objective is to estimate $\bbx(t) \in\mathbb{R}^{n}$, given the history of noise corrupted sensor data up to time $t$. Given (\ref{jkoSDE}) and (\ref{ObservationSDE}), solving the filtering problem amounts to computing the posterior PDF $\rho^{+}(\bbx(t),t)$ that obeys the Kushner-Stratonovich stochastic PDE \cite{Stratonovich1960, Kushner1964,FrostKaliath1971PartIII,FKK1972}
%\begin{align}
%&\mathrm{d}\rho^{+} = \left[\mathcal{L}_{\rm{FP}}\:\mathrm{d}t  + \left(\bbc(\bbx(t),t) - \mathbb{E}_{\rho^{+}}\{\bbc(\bbx(t),t)\}\right)^{\top} \bbR^{-1} \right.\nonumber\\
%&\left.\left(\mathrm{d}\bbz(t) - \mathbb{E}_{\rho^{+}}\{\bbc(\bbx(t),t)\}\mathrm{d}t\right)\right]\: \rho^{+},
%\label{KSpde}	
%\end{align}
%where $\mathcal{L}_{\mathrm{FP}}$ is the Fokker-Planck operator associated with (\ref{jkoSDE}), given by (\ref{eq:FP}).

In this section, we focus on the linear Gaussian filtering problem, with process model and measurement models
\begin{align*}
\mathrm{d}\bbx(t) &= \bbA\bbx(t)\mathrm{d}t + \sqrt{2}\bbB\mathrm{d}\bbw(t),\\
\mathrm{d}\bbz(t) &= \bbC\bm{x}(t)\:\mathrm{d}t + \mathrm{d}\bbv(t),
\end{align*}
where $\bbC\in\mathbb{R}^{m\times n}$, and $\rho_{0} = \mathcal{N}(\bbmu_{0},\bbP_{0})$. The conditional PDF $\rho^{+}(\bbx(t),t) = \mathcal{N}(\bbmu^{+}(t),\bbP^{+}(t))$, given measurements up to time $t$, is well-known and given by the \textbf{Kalman-Bucy filter} \cite{KalmanBucy1961}
\begin{subequations}\label{KalmanBucy}
\begin{align}
&\mathrm{d}\bbmu^{+}(t) = \bbA\bbmu^{+}(t)\mathrm{d}t + \bbK(t)\left(\mathrm{d}\bbz(t) - \bbC\bbmu^{+}(t)\mathrm{d}t\right), \label{eq:KB1}\\
& \! \!\dot{\bbP}^{+} \!(t)  \!=  \!\bbA\bbP^{+} \!(t) \!+  \!\! \bbP^{+} \!(t)\bbA^{\top}  \! \!+ \! 2\bbB\bbB^{\top}    \!\!\!-  \!\bbK \!(t)\bbR\bbK \!(t)^{\top} \label{eq:KB2}
\end{align}
\end{subequations}
that specifies a vector SDE and a matrix ODE, respectively, for 
 the conditional mean $\bbmu^{+}(t)$ and covariance $\bbP^{+}(t)$.
The initial conditions are $\bbmu^{+}(0) = \bbmu_{0}$, $\bbP^{+}(0) = \bbP_{0}$, and $\bbK(t) := \bbP^{+}(t)\bbC^{\top}\bbR^{-1}$ is the so-called Kalman gain.

%The resulting estimator is referred as , and is optimal in the sense that $\bbmu^{+}(t)$ is the minimum error variance as well as the maximum Bayesian \emph{a posteriori} (MAP) estimate of the state $\bbx(t)$. 
In the sequel, we demonstrate that by applying the two-step optimization strategy we used before in Section \ref{UncPropSectionLabel}, we can recover the Kalman-Bucy filter from LMMR-equation (\ref{KLvariational}) for the linear Gaussian case as the $h\downarrow 0$ limit.

\subsection{LMMR gradient descent scheme}\label{PrashantSubsection}

Once again we proceed with carrying out the following two optimization steps. First, we optimize \eqref{KLvariational} over $\cD_{\bm{\mu},\bm{P}}$, and then optimize the minimum value over the choice of parameters $({\bm{\mu},\bm{P}})$.

\subsubsection{Optimizing over $\cD_{\bm{\mu},\bm{P}}$}\label{FirstStepLaugesen}
Consider $\varrho_{k}^{-} = \mathcal{N}(\bbmu_{k}^{-},\bbP_{k}^{-})$ to be our prior for the state PDF at time $t=kh$. Observe that 
{\small{\begin{align}
\underset{\varrho\in\mathscr{D}_{\bbmu,\bbP}}{\inf}D_{\mathrm{KL}}\left(\varrho\|\mathcal{N}(\bbmu_{k}^{-},\bbP_{k}^{-})\right) &= \underset{\varrho\in\mathscr{D}_{\bbmu,\bbP}}{\inf} \left[\int_{\mathbb{R}^{n}} \varrho(\bbx)\log\varrho(\bbx)\mathrm{d}\bbx \right.\nonumber\\
&\left. - \mathbb{E}_{\varrho}\{\log\mathcal{N}(\bbmu_{k}^{-},\bbP_{k}^{-})\}\right],
\label{DKLalone}	
\end{align}}}

\noindent
and that
{\small{\begin{align*}
\mathbb{E}_{\varrho}\{\log\mathcal{N}(\bbmu_{k}^{-},\bbP_{k}^{-})\}
= -\frac{1}{2}\left[(\bbmu - \bbmu_{k}^{-})^{\top}\left(\bbP_{k}^{-}\right)^{-1}(\bbmu - \bbmu_{k}^{-})\right.\\\left.
+ \tr\left(\bbP(\bbP_{k}^{-})^{-1}\right)\right] - \frac{1}{2}\log\left((2\pi)^{n}\det(\bbP_{k}^{-})\right)\end{align*}}}

\noindent remains invariant for all $\varrho\in\mathscr{D}_{\bbmu,\bbP}$. Therefore, the arginf in (\ref{DKLalone}) is achieved by the Gaussian PDF $\mathcal{N}(\bbmu,\bbP)$ (i.e., the maximum entropy PDF with given mean-covariance), and the infimal value is precisely $D_{\rm{KL}}(\mathcal{N}(\bbmu,\bbP)\|\mathcal{N}(\bbmu_{k}^{-},\bbP_{k}^{-}))$. On the other hand, notice that

{\small{\begin{align}
&\underset{\varrho\in\mathscr{D}_{\bbmu,\bbP}}{\inf}\frac{1}{2}\:\mathbb{E}_{\varrho}\{(\bby_{k} - \bbC\bbx)^{\top} \bbR^{-1} (\bby_{k} - \bbC\bbx)\} = \frac{1}{2}\left[(\bby_{k} - \bbC\bbmu)^{\top} \right.\nonumber\\
&\left.\bbR^{-1}(\bby_{k} - \bbC\bbmu)+\tr\left(\bbC^{\top}\bbR^{-1}\bbC\bbP\right)\right] = \text{constant}
\label{SurpriseAlone}	
\end{align}}}

\noindent
as well over $\cD_{\bm{\mu},\bm{P}}$.
Hence
\begin{align*}
&\underset{\varrho\in\mathscr{D}_{\bbmu,\bbP}}\arginf \left[D_{\mathrm{KL}}\left(\varrho\|\mathcal{N}(\bbmu_{k}^{-},\bbP_{k}^{-})\right)\right.\\
& \left.\hspace*{1cm}+ \frac{h}{2}\:\mathbb{E}_{\varrho}\{(\bby_{k}  - \bbC\bbx)^{\top} \bbR^{-1} (\bby_{k} - \bbC\bbx)\}\right] = \mathcal{N}(\bbmu,\bbP),
\end{align*}
and the corresponding infimum value is
\begin{align}
&\frac{1}{2}\left[\tr\left((\bbP_{k}^{-})^{-1}\bbP\right) + (\bbmu_{k}^{-}-\bbmu)^{\top}(\bbP_{k}^{-})^{-1}(\bbmu_{k}^{-}-\bbmu) - n -\right.\nonumber\\
&\left.\log\det\left((\bbP_{k}^{-})^{-1}\bbP\right)\right] \: + \: \frac{h}{2}\left[(\bby_{k} - \bbC\bbmu)^{\top} \bbR^{-1} (\bby_{k} - \bbC\bbmu)\right.\nonumber\\
&\left.+\tr\left(\bbC^{\top}\bbR^{-1}\bbC\bbP\right)\right].
\label{InfValueDKLstep1}	
\end{align}

\subsubsection{Optimizing over $(\bm{\mu},\bm{P})$}
Equating the partial derivative of (\ref{InfValueDKLstep1}) w.r.t. $\bbmu$ to zero, and setting $\bbmu\equiv\bbmu_{k}^{+}$ in the resulting algebraic equation, we get
\begin{align}
&(\bbP_{k}^{-})^{-1}\left(\bbmu_{k}^{-} - \bbmu_{k}^{+}\right) + h\bbC^{\top}\bbR^{-1}\left(\bby_{k} - \bbC\bbmu_{k}^{+}\right) = \bm{0}, \nonumber\\
\Rightarrow &\bbmu_{k}^{+} = \bbmu_{k}^{-} + h\bbP_{k}^{-}\bbC^{\top}\bbR^{-1}\left(\bby_{k}-\bbC\bbmu_{k}^{+}\right).
\label{IntermediateKalmanBucyMeanSDE}	
\end{align}
On the other hand, equating the partial derivative of (\ref{InfValueDKLstep1}) w.r.t. $\bbP$ to zero, and then setting $\bbP\equiv\bbP_{k}^{+}$ in the resulting algebraic equation, we get
{\small{\begin{align}
&(\bbP_{k}^{+})^{-1} = (\bbP_{k}^{-})^{-1} + h\bbC^{\top}\bbR^{-1}\bbC \Rightarrow \bbP_{k}^{+} = \left(\bbI + h\bbP_{k}^{-}\bbC^{\top}\right.\nonumber\\
&\left.\bbR^{-1}\bbC\right)^{-1}\bbP_{k}^{-} = \bbP_{k}^{-} - h\bbP_{k}^{-}\bbC^{\top}\bbR^{-1}\bbC\bbP_{k}^{-} + O(h^{2}).
\label{IntermediateKalmanBucyCovODE}	
\end{align}}}
\hspace*{-6pt}
With $\Delta\bbz_{k} = \bby_{k}h$, as in Section I,
\[
\mathrm{d}\bbz(t) = \Delta\bbz_{k} + O(h^{2}),
\]
\[
\bbmu^{+}(t)\mathrm{d}t = \bbmu^{+}_{k} h + O(h^{2}),
\]
 and from (\ref{JKOmeanrecoveryGeneralcase}) that 
 \[
 \bbmu_{k}^{-} = (\bbI + h\bbA)\bbmu_{k-1}^{+} + O(h^{2})
.
\]
 These, together with (\ref{IntermediateKalmanBucyCovODE}), allow us to simplify (\ref{IntermediateKalmanBucyMeanSDE}) as
{\small{\begin{eqnarray*}
\bbmu_{k}^{+} \!\!-\! \bbmu_{k-1}^{+} = h\bbA\bbmu_{k-1}^{+} + \bbP_{k}^{+}\bbC^{\top}\bbR^{-1}\left(\Delta\bbz_{k} - h\bbC\bbmu_{k}^{+}\right) + O(h^{2}),
%\label{RecoverKalmanBucyMeanSDE}	
\end{eqnarray*}}}
\hspace*{-6pt}
which in the limit $h\downarrow 0$, leads to \eqref{eq:KB1}.

Substituting \eqref{eq:unproven} into (\ref{IntermediateKalmanBucyCovODE}) we arrive at
\begin{eqnarray}
\bbP_{k}^{+} - \bbP_{k-1}^{+} = h(\bbA\bbP_{k-1}^{+} + \bbP_{k-1}^{+}\bbA^{\top} + 2\bbB\bbB^{\top}) \nonumber\\
- h\bbP_{k-1}^{+}\bbC^{\top}\bbR^{-1}\bbC\bbP_{k-1}^{+} \:+\: O(h^{2}).
\label{RecoverKalmanBucyCovODE}		
\end{eqnarray}
In the limit $h\downarrow 0$, (\ref{RecoverKalmanBucyCovODE}) recovers  \eqref{eq:KB2}.

\subsection{Alternative JKO-like schemes for filtering}
The ideas in the LMMR-scheme suggest the possibility of alternative variational schemes to approximate stochastic estimators. Such a viewpoint has been put forth in  \cite{YezziVerriest2007}, promoting the notion of regularized dynamic inversion.
As an example, one may consider a {\em gradient descent with respect to the Wasserstein} distance $\frac{1}{2}W_{2}^{2}$, instead of KL-divergence $D_{\rm{KL}}$ in (\ref{KLvariational}).
%\subsection{A Gradient Descent Scheme with respect to $W_{2}^{2}$}\label{StochasticObserverSubsection}
%In this subsection, we replace the distance measure $D_{\rm{KL}}$ in (\ref{KLvariational}) with $\frac{1}{2}W_{2}^{2}$, that is, we consider the recursion 
In that case, the posterior may be constructed according to
\begin{eqnarray}
\varrho_{k}^{+}(\bbx,h) = \underset{\varrho\in\mathcal{D}_{2}}\arginf  \: \frac{1}{2}W_{2}^{2}\left(\varrho,\varrho_{k}^{-}\right) \: + \: h\Phi(\varrho), \quad k\in\mathbb{N},
\label{GnericWassersteinJKO}	
\end{eqnarray}
where the functional $\Phi(\cdot)$ is as in \eqref{eq:Phi}. The template of the two-step optimization again applies and, specializing to the linear Gaussian case, the solution of (\ref{GnericWassersteinJKO}) in the $h\downarrow 0$ limit, is $\mathcal{N}(\bbmu^{+}(t),\bbP^{+}(t))$, given by
\begin{subequations}\label{ObserverSDEODE}
\begin{align}
&\!\!\!\mathrm{d}\bbmu^{+}(t) = \bbA\bbmu^{+}(t)\mathrm{d}t + \bbL\left(\mathrm{d}\bbz(t) - \bbC\bbmu^{+}(t)\mathrm{d}t\right),\label{ObserverMean}\\
&\!\!\!\!\dot{\bbP}^{+}(t) \!\!= \!\!(\bbA - \bbL\bbC)\bbP^{+}(t) \!+\! \bbP^{+}(t)(\bbA - \bbL\bbC)^{\top} \!\!\!+\! 2\bbB\bbB^{\top}\!\!\!\!\!\label{ObserverCov}
\end{align}
\end{subequations}
where $\bbL := \bbC^{\top}\bbR^{-1}$, and $\bbmu^{+}(0) = \bbmu_{0}$, $\bbP^{+}(0) = \bbP_{0}$. This follows by noticing from Sections \ref{FirstStepSpecialCase} and \ref{FirstStepLaugesen} that
\[
\underset{\varrho\in\mathcal{D}_{\bbmu,\bbP}}\arginf  \: \left[\frac{1}{2}W_{2}^{2}\left(\varrho,\mathcal{N}(\bbmu_{k}^{-},\bbP_{k}^{-}\right)+h\Phi(\varrho)\right] = \mathcal{N}(\bbmu,\bbP),
\]
where the infimum value is
\begin{eqnarray}
&&\!\!\!\!\!\!\!\!\!\!\!\!\frac{1}{2}\left[\parallel\bm{\mu}\!-\!\bm{\mu}_{k}^{-}\parallel_{2}^{2} + \tr\left(\bm{P} + \bm{P}_{k}^{-}\! -\! 2 \left((\bm{P}_{k}^{-})^{\frac{1}{2}} \bm{P} (\bm{P}_{k}^{-})^{\frac{1}{2}}\right)^{\frac{1}{2}}\right)\right] \nonumber\\
&&\!\!\!\!\!\!\!\!\!\!\!\!\!+ \frac{h}{2}\left[(\bby_{k} - \bbC\bbmu)^{\top} \bbR^{-1} (\bby_{k} - \bbC\bbmu)\!+\!\tr\left(\bbC^{\top}\bbR^{-1}\bbC\bbP\right)\right].
\label{InfValueStocObserver}	
\end{eqnarray}
Equating the partial derivative of (\ref{InfValueStocObserver}) w.r.t. $\bbmu$ to zero, then setting $\bbmu\equiv\bbmu_{k}^{+}$, and using (\ref{JKOmeanrecoveryGeneralcase}), we find $(\bbmu_{k}^{+} - \bbmu_{k-1}^{+})$ equals
\begin{eqnarray}
h\bbA\bbmu_{k-1}^{+} + \bbC^{\top}\bbR^{-1}\left(\Delta\bbz_{k} - h\bbC\bbmu_{k}^{+}\right) + O(h^{2}),
\label{RecoveringMuStocObserver}	
\end{eqnarray}
which in the limit $h\downarrow 0$, results the SDE (\ref{ObserverMean}). Similarly, using Lemma \ref{MatrixDerivativeofWassCrossTermLemma}, we equate the partial derivative of (\ref{InfValueStocObserver}) w.r.t. $\bbP$ to zero, and then setting $\bbP\equiv\bbP_{k}^{+}$, we get
\begin{align*}
&(\bbP_{k}^{+})^{-1} = \left(\bbI + h\bbC^{\top}\bbR^{-1}\bbC\right)(\bbP_{k}^{-})^{-1}\left(\bbI + h\bbC^{\top}\bbR^{-1}\bbC\right)\Rightarrow\\ 
&\bbP_{k}^{+} = \bbP_{k}^{-} - h\left(\bbP_{k}^{-}\bbC^{\top}\bbR^{-1}\bbC + \bbC^{\top}\bbR^{-1}\bbC\bbP_{k}^{-}\right) + O(h^{2}),		
\end{align*}
which combined with the recursion $\bbP_{k}^{-} = \bbP_{k-1}^{+} + h(\bbA\bbP_{k-1}^{+} \allowbreak + \bbP_{k-1}^{+}\bbA^{\top} + 2\bbB\bbB^{\top}) + O(h^{2})$ from Section \ref{GenericLinFPKrecov}, yields
\begin{align}
\bbP_{k}^{+} = \bbP_{k-1}^{+} + &h\left[\left(\bbA - \bbC^{\top}\bbR^{-1}\bbC\right)\bbP_{k-1}^{+} + \bbP_{k-1}^{+}\left(\bbA - \right.\right.\nonumber\\ 
&\left.\left.\bbC^{\top}\bbR^{-1}\bbC\right)^{\top} + 2\bbB\bbB^{\top}\right] + O(h^{2}).
\label{RecoveringPStocObserver}	
\end{align}
In the limit $h\downarrow 0$, recursion (\ref{RecoveringPStocObserver}) gives Lyapunov ODE (\ref{ObserverCov}).

It is instructive to compare the SDE-ODE system (\ref{ObserverSDEODE}) with that in (\ref{KalmanBucy}). In the case of (\ref{ObserverSDEODE}), the estimator is of a Luenberger type with a static gain matrix $\bbL$ which is decoupled from the covariance, unlike (\ref{KalmanBucy}). The estimator (\ref{ObserverSDEODE}) is obviously not optimal in the minimum mean-square error sense. It is only presented here as a guideline to explore other variational schemes with desirable properties.

\section{Concluding remarks} 

Reformulating uncertainty propagation and the filtering equations as gradient flows \cite{AmbrosioBook2008} is 
potentially transformative \cite{JKO1998} \cite{LaugesenMehta2015}. The full power of this viewpoint is yet to be uncovered. Moreover, casting the iterative approximation steps in the language of proximal operators on the space of density functions may provide theoretical insights and computational benefits. A specific direction of future work would be developing proximal algorithms \cite{ParikhBoyd2013} to numerically solve the nonlinear filtering problem by recursively solving convex optimization problems, and to quantify computational performance of the same with respect to existing sequential Monte Carlo algorithms like the particle filter. The purpose of the present paper has been to highlight and elucidate the ideas in \cite{JKO1998} and \cite{LaugesenMehta2015} in the context of linear Gaussian systems. We hope that this study will help to motivate further exploration of this topic. 

%%%%%%%%%%%%%%%%%%%%%%%%%%%%%%%%%%%%%%%%%%%%%%%%%%%%%%%%%%%%%%%%%%%%%%%%%%%%%%%%%%%%%%%%%%%%%%%%%%%%%%%%%%%%%%%%%%%%%%%%%%%%%%%%%%%%%%%

\appendix
In this Appendix, we collect some lemmas that are used in Sections \ref{UncPropSectionLabel} and \ref{FilteringSectionLabel}. In addition, we will show in Corollary \ref{RecoverCarlenGangboThm3.1} below that applying Lemma \ref{TraceInequalityLemma} and \ref{WsqareProjLemma} together enables us to provide an alternative proof of a Theorem in \cite{CarlenGangbo2003}, which might be of independent interest.

\begin{lemma}\label{TraceInequalityLemma}
If $\bm{X}$ and $\bm{Y}$ are symmetric positive definite matrices, then $\tr\left(\bm{X}^{\frac{1}{2}}\bm{Y}\bm{X}^{\frac{1}{2}}\right)^{\frac{1}{2}} \leq \sqrt{\tr\left(\bm{X}\right)\:\tr\left(\bm{Y}\right)}$. 
\end{lemma}
\begin{proof}
From Uhlmann's variational formula (see \cite{Uhlmann1976}, also Theorem 6.1 in \cite{PetzBook2008})	, given any $\bm{G}\succ \bm{0}$, we have
\begin{eqnarray}
\tr\left(\left(\bm{X}^{\frac{1}{2}}\bm{Y}\bm{X}^{\frac{1}{2}}\right)^{\frac{1}{2}}\right) \leq \sqrt{\tr\left(\bm{XG}\right)\:\tr\left(\bm{YG}^{-1}\right)},
\label{proofTraceInequality}	
\end{eqnarray}
where the equality in (\ref{proofTraceInequality}) is achieved for the specific choice $\bm{G}_{\mathrm{opt}} = \bm{Y}^{\frac{1}{2}}\left(\bm{X}^{\frac{1}{2}}\bm{Y}\bm{X}^{\frac{1}{2}}\right)^{-\frac{1}{2}}\bm{X}^{\frac{1}{2}}\bm{Y}^{\frac{1}{2}}\bm{X}^{-\frac{1}{2}}$. Specializing (\ref{proofTraceInequality}) for $\bm{G} = \bm{Y}$, and noting that $\tr\left(\bm{X}\bm{Y}\right) \leq \tr\left(\bm{X}\right)\tr\left(\bm{Y}\right)$, the statement follows. 
\end{proof}

\begin{lemma}\label{WsqareProjLemma}
Given a PDF $\varrho_{0}(\bm{x})\in\cD_{2}$ with mean $\bm{\mu}_{0}\in\mathbb{R}^{n}$, and $n\times n$ covariance matrix $\bm{P}_{0} \succ \bm{0}$. Then $\underset{\varrho\in\cD_{\bm{\mu},\bm{P}}}{\inf} \; W_{2}^{2}\left(\varrho,\varrho_{0}\right)$ equals
	\begin{eqnarray}
	 \parallel\bm{\mu}-\bm{\mu}_{0}\parallel_{2}^{2} \: + \: \tr\left(\bm{P} + \bm{P}_{0} - 2 \left(\bm{P}_{0}^{\frac{1}{2}} \bm{P} \bm{P}_{0}^{\frac{1}{2}}\right)^{\frac{1}{2}}\right),
	\label{infWsquare}	
	\end{eqnarray}
and is achieved by push-forward of $\varrho_{0}(\bm{x})$ via an affine transport map $\bm{M}\bm{x} + \bm{m}$, where $\bm{M} := \bm{P}^{\frac{1}{2}} \left(\bm{P}^{\frac{1}{2}}\bm{P}_{0}\bm{P}^{\frac{1}{2}}\right)^{-\frac{1}{2}} \bm{P}^{\frac{1}{2}}$, and $\bm{m}:=\bm{\mu}-\bm{\mu}_{0}$, that is, the $\mathrm{arginf}$ for (\ref{infWsquare}) is $\varrho(\bm{x}) = \sqrt{\frac{\det(\bm{P}_{0})}{\det(\bm{P})}}\:\varrho_{0}\left( \bm{P}^{-\frac{1}{2}} \left(\bm{P}^{\frac{1}{2}}\bm{P}_{0}\bm{P}^{\frac{1}{2}}\right)^{\frac{1}{2}} \bm{P}^{-\frac{1}{2}}\left(\bm{x} - \bm{\mu}\right) + \bm{\mu}_{0}\right)$. In particular, if $\varrho_{0} = \mathcal{N}\left(\bm{\mu}_{0},\bm{P}_{0}\right)$, then $\varrho = \mathcal{N}\left(\bm{\mu},\bm{P}\right)$.
\end{lemma}
\begin{proof}
Let $\varrho_{0}$ be as given, and choose any $\varrho\in\cD_{\bm{\mu},\bm{P}}$. Let $\overline{\varrho}_{0}$ and $\overline{\varrho}$ be obtained by translating $\varrho_{0}$ and $\varrho$ respectively, such that both $\overline{\varrho}_{0}$ and $\overline{\varrho}$ have zero mean. Using (\ref{WassDefn}), we can directly verify \cite[p. 236]{GivensShortt1984} that $W_{2}^{2}\left(\varrho,\varrho_{0}\right) = \parallel\bm{\mu}-\bm{\mu}_{0}\parallel_{2}^{2} + W_{2}^{2}\left(\overline{\varrho},\overline{\varrho}_{0}\right)$. On the other hand, it is known \cite[p. 11, Proposition 1.1.6]{RachevRuschendorfBook1998} that 
\begin{align}
&W_{2}^{2}\left(\overline{\varrho},\overline{\varrho}_{0}\right) \geq \tr\left(\bm{P} + \bm{P}_{0} - 2 \left(\bm{P}_{0}^{\frac{1}{2}} \bm{P} \bm{P}_{0}^{\frac{1}{2}}\right)^{\frac{1}{2}}\right) \nonumber\\
&\Rightarrow W_{2}^{2}\left(\varrho,\varrho_{0}\right) \geq \;\text{ right hand side of (\ref{infWsquare})}.
\label{StillAnInequality}	
\end{align}
Now consider a candidate transport map $\bm{M}\bm{x} + \bm{m}$ where $\bm{M}$ and $\bm{m}$ are functions of $\bm{P},\bm{P}_{0},\bm{\mu},\bm{\mu}_{0}$ as in the statement. It suffices to prove that our candidate transport map indeed achieves the equality in (\ref{StillAnInequality}). To this end, directly substituting the expressions for $\bm{M}$ and $\bm{m}$, notice that the push-forward has mean $\bm{M}\bm{\mu}_{0} + \bm{m} = \bm{\mu}$, and covariance $\bm{M}\bm{P}_{0}\bm{M}^{\top} = \bm{P}$. So our candidate transport map $(\bm{M},\bm{m})$ is feasible. To show optimality, from (\ref{WassDefn}) notice that $W_{2}^{2}\left(\overline{\varrho},\overline{\varrho}_{0}\right) = \underset{\bm{C}\in\mathbb{R}^{d\times d}}{\inf}\:\tr(\bm{P}+\bm{P}_{0} - 2\bm{C})$, where $\bm{C}:=\bm{M}\bm{P}_{0}$ solves $\bm{P}_{0} - \bm{C}\bm{P}^{-1}\bm{C}^{\top} \succeq \bm{0}$, which has known optimal solution $\bm{C}_{\mathrm{opt}}:=\bm{M}_{\mathrm{opt}}\bm{P}_{0}=\bm{P}_{0}\bm{P}^{\frac{1}{2}} \left(\bm{P}^{\frac{1}{2}}\bm{P}_{0}\bm{P}^{\frac{1}{2}}\right)^{-\frac{1}{2}} \bm{P}^{\frac{1}{2}}$. Since our candidate $\bm{M}:=\bm{P}^{\frac{1}{2}} \left(\bm{P}^{\frac{1}{2}}\bm{P}_{0}\bm{P}^{\frac{1}{2}}\right)^{-\frac{1}{2}} \bm{P}^{\frac{1}{2}}$ satisfies $\tr\left(\bm{M}\bm{P}_{0}\right) = \tr\left(\bm{M}_{\mathrm{opt}}\bm{P}_{0}\right) = \tr\left(\left(\bm{P}_{0}^{\frac{1}{2}} \bm{P} \bm{P}_{0}^{\frac{1}{2}}\right)^{\frac{1}{2}}\right)$, the statement follows.
\end{proof}

%We will use Lemma \ref{WsqareProjLemma} in the next subsection for computing $\rho(\bm{x}(t),t)$ for a linear system via JKO scheme. 
In the Corollary below, combining Lemma \ref{TraceInequalityLemma} and \ref{WsqareProjLemma}, we recover a result in \cite[Theorem 3.1]{CarlenGangbo2003}. %{\textcolor{red}{How to politely phrase a sentence saying our proof is more direct than \cite{CarlenGangbo2003}?}

\begin{corollary}\label{RecoverCarlenGangboThm3.1}
	Given $d$-dimensional joint PDF $\varrho_{0}$ with mean $\bm{\mu}_{0}$, covariance $\bm{P}_{0} \succ \bm{0}$, suppose $\tr(\bm{P}_{0}) = \tau_{0}$. For fixed $\bm{\mu}$ and $\tau>0$,
	\begin{eqnarray}
	\underset{\varrho\in\cD_{\bm{\mu},\tau}}{\inf} \: W_{2}^{2}\left(\varrho,\varrho_{0}\right) = \left(\sqrt{\tau} - \sqrt{\tau}_{0}\right)^{2} \: + \: \parallel \bm{\mu} - \bm{\mu}_{0} \parallel_{2}^{2}, 
	\label{WassCarlenGangbo}	
	\end{eqnarray}
and is achieved by $\varrho(\bm{x}) = \left(\frac{\tau_{0}}{\tau}\right)^{\frac{d}{2}} \varrho_{0}\left(\frac{\tau_{0}}{\tau}\left(\bm{x} - \bm{\mu}\right) + \bm{\mu}_{0}\right)$.
\end{corollary}
\begin{proof}
Let us choose $\bm{P}:=\frac{\tau}{\tau_{0}}\bm{P}_{0}$, and from (\ref{infWsquare}) observe that $\underset{\xi\in\cD_{\bm{\mu},\bm{P}}}{\inf}\: W_{2}^{2}\left(\xi,\varrho_{0}\right) = \left(\sqrt{\tau} - \sqrt{\tau}_{0}\right)^{2} \: + \: \parallel \bm{\mu} - \bm{\mu}_{0} \parallel_{2}^{2}$.	On the other hand, for any $\varrho\in \cD_{\bm{\mu},\tau}$, we know from (\ref{StillAnInequality}) that 
\begin{eqnarray*}
W_{2}^{2}\left(\varrho,\varrho_{0}\right) \geq \: \tau + \tau_{0} - 2\:\tr\left(\bm{P}_{0}^{\frac{1}{2}} \bm{S} \bm{P}_{0}^{\frac{1}{2}}\right)^{\frac{1}{2}} \: + \: \parallel \bm{\mu} - \bm{\mu}_{0} \parallel_{2}^{2},
%\label{IntermediateInequality}
\end{eqnarray*}
where $\bm{S}$ is the covariance of $\varrho$. Using Lemma \ref{TraceInequalityLemma}, we get
\begin{eqnarray}
\tr\left(\bm{P}_{0}^{\frac{1}{2}} \bm{S} \bm{P}_{0}^{\frac{1}{2}}\right)^{\frac{1}{2}} \leq \sqrt{\tau \tau_{0}} \Rightarrow W_{2}^{2}\left(\varrho,\varrho_{0}\right) \geq \nonumber\\
\underset{\xi\in\cD_{\bm{\mu},\bm{P}}}{\inf}\: W_{2}^{2}\left(\xi,\varrho_{0}\right) = \left(\sqrt{\tau} - \sqrt{\tau}_{0}\right)^{2} \: + \: \parallel \bm{\mu} - \bm{\mu}_{0} \parallel_{2}^{2},
\label{CorollaryFinal}	
\end{eqnarray}
and that the equality is achieved when $\bm{S} = \bm{P} = \frac{\tau}{\tau_{0}}\bm{P}_{0}$. In that case, $\det(\bm{S})=\left(\frac{\tau}{\tau_{0}}\right)^{d}\det(\bm{P}_{0})$, and hence Lemma \ref{WsqareProjLemma} yields the arg inf $\varrho(\bm{x})$ for (\ref{WassCarlenGangbo}) as
 \begin{align*}
&\sqrt{\frac{\det(\bm{P}_{0})}{\det(\bm{S})}}\:\varrho_{0}\left(\bm{S}^{-\frac{1}{2}} \left(\bm{S}^{\frac{1}{2}}\bm{P}_{0}\bm{S}^{\frac{1}{2}}\right)^{\frac{1}{2}} \bm{S}^{-\frac{1}{2}}\left(\bm{x} - \bm{\mu}\right) + \bm{\mu}_{0}\right) \\
&= \left(\frac{\tau_{0}}{\tau}\right)^{\frac{d}{2}} \varrho_{0}\left(\frac{\tau_{0}}{\tau}\left(\bm{x} - \bm{\mu}\right) + \bm{\mu}_{0}\right).
 \end{align*}
\vspace*{-0.1in} 
\end{proof}

\begin{lemma}\label{infFreeEnergyLemma}
If $\mathcal{E}(\cdot)$ depends on $\varrho$ only via the mean and covariance of $\varrho$, then $\underset{\varrho\in\cD_{\bm{\mu},\bm{P}}}{\inf} \; \mathcal{F}\left(\varrho\right)$ is achieved by $\mathcal{N}\left(\bm{\mu},\bm{P}\right)$.	
\end{lemma}
\begin{proof}
As $\mathcal{E}(\varrho)\equiv\mathcal{E}(\bm{\mu},\bm{P})$, hence from (\ref{FreeEnergyDefn}) we get $\underset{\varrho\in\cD_{\bm{\mu},\bm{P}}}{\inf} \mathcal{F}\left(\varrho\right) = \mathcal{E}(\bm{\mu},\bm{P}) + \beta^{-1}\underset{\varrho\in\cD_{\bm{\mu},\bm{P}}}{\inf} \int \varrho\log\varrho\:\mathrm{d}\bm{x}$. Since $\mathcal{N}(\bm{\mu},\bm{P})$ is the maximum entropy PDF under prescribed mean $\bm{\mu}$ and covariance $\bm{P}$, hence the statement.
\end{proof}

\begin{lemma}\label{MatrixDerivativeofWassCrossTermLemma}
For $\bm{P},\bm{P}_{0}\succ 0$,
\begin{eqnarray}\nonumber
\displaystyle\frac{\partial}{\partial\bm{P}}\:\tr\left(\bm{P}_{0}^{\frac{1}{2}} \bm{P} \bm{P}_{0}^{\frac{1}{2}}\right)^{\frac{1}{2}} = \displaystyle\frac{1}{2}\bm{P}_{0}^{\frac{1}{2}}\left(\bm{P}_{0}^{-\frac{1}{2}} \bm{P}^{-1} \bm{P}_{0}^{-\frac{1}{2}}\right)^{\frac{1}{2}}\bm{P}_{0}^{\frac{1}{2}}.
\label{MatrixDerivativeofWassCrossTerm}	
\end{eqnarray}
\end{lemma}
\begin{proof}
We refer the readers to Appendix B in \cite{HalderWendelACC2016}.	
\end{proof}

%%%%%%%%%%%%%%%%%%%%%%%%%%%%%%%%%%%%%%%%%%%%%%%%%%%%%%%%%%%%%%%%%%%%%%%%%%%%%%%%%%%%%%%%%%%%%%%%%%%%%%%%%%%%%%%%%%%%%%%%%%%%%%%%%%%%%%%


\begin{thebibliography}{99}%


\bibitem{BauschkeCombettes2011}
H.H.~Bauschke, and P.L.~Combettes, \emph{Convex Analysis and Monotone Operator Theory in Hilbert Spaces}. CMS Books in Mathematics, Springer; 2011.


\bibitem{ParikhBoyd2013}
N.~Parikh, and S.~Boyd, ``Proximal Algorithms". \emph{Foundations and Trends in Optimization}. Vol. 1, No. 3, pp. 127--239, 2014.


%%%%%%%%%%%%%%%%%%%%%%%%%%%%%%%%%%%%%%%%%%%%%%%%%%%%%%%%%%%%%%%%%%%%%%%%%%%%%%%%%%%%%%%


\bibitem{JKO1998}
R.~Jordan, D.~Kinderlehrer, and F.~Otto, ``The Variational Formulation of the Fokker--Planck Equation". \emph{SIAM Journal on Mathematical Analysis}. Vol. 29, No. 1, pp. 1--17, 1998.

%%%%%%%%%%%%%%%%%%%%%%%%%%%%%%%%%%%%%%%%%%%%%%%%%%%%%%%%%%%%%%%%%%%%%%%%%%%%%%%%%%%%%%%

\bibitem{VillaniBook2003}
C.~Villani, \emph{Topics in Optimal Transportation}. Graduate Studies in
Mathematics, Vol. 58, First ed., American Mathematical Society; 2003.

\bibitem{French}
J.-D.~ Benamou, and Y.~Brenier,
``A Computational Fluid Mechanics Solution to the
Monge--Kantorovich Mass Transfer Problem''. \emph{Numerische Mathematik}. Vol. 84, No. 3, pp. 375--393, 2000.
\balance

%%%%%%%%%%%%%%%%%%%%%%%%%%%%%%%%%%%%%%%%%%%%%%%%%%%%%%%%%%%%%%%%%%%%%%%%%%%%%%%%%%%%%%%

\bibitem{AstromBook1970}
K.J.~\AA str\"om, \emph{Introduction to Stochastic Control Theory}. Academic Press; 1970.



%%%%%%%%%%%%%%%%%%%%%%%%%%%%%%%%%%%%%%%%%%%%%%%%%%%%%%%%%%%%%%%%%%%%%%%%%%%%%%%%%%%%%%%

\bibitem{RiskenBook1989}
H.~Risken, \emph{The Fokker-Planck Equation: Methods of Solution and Applications}. Springer Series in Synergetics, Vol. 18, First ed., Springer; 1989.

%\bibitem{SoizeBook1994}
%C.~Soize, \emph{The Fokker-Planck Equation for Stochastic Dynamical Systems and Its Explicit Steady State Solutions}. Series on Advances in Mathematics for Applied Sciences, Vol. 17, First ed., World Scientific; 1994.


%%%%%%%%%%%%%%%%%%%%%%%%%%%%%%%%%%%%%%%%%%%%%%%%%%%%%%%%%%%%%%%%%%%%%%%%%%%%%%%%%%%%%%%

\bibitem{AmbrosioBook2008}
L.~Ambrosio, N.~Gigli, and G.~Savar{\'e}, \emph{Gradient Flows: in Metric Spaces and in the Space of Probability Measures}. Lectures in Mathematics, ETH Z\"{u}rich, Second ed., Birkh\"{a}user; 2008.

\bibitem{LaugesenMehta2015}
R.S.~Laugesen, P.G.~Mehta, S.P.~Meyn, and M.~Raginsky, ``Poisson's Equation in Nonlinear Filtering". \emph{SIAM Journal on Control and Optimization}. Vol. 53, No. 1, pp. 501--525, 2015.




%%%%%%%%%%%%%%%%%%%%%%%%%%%%%%%%%%%%%%%%%%%%%%%%%%%%%%%%%%%%%%%%%%%%%%%%%%%%%%%%%%%%%%%

\bibitem{Uhlmann1976}
A.~Uhlmann, ``The ``Transition Probability" in the State Space of a *-Algebra". \emph{Reports on Mathematical Physics}, Vol. 9, No. 2, pp. 273--279, 1976.

\bibitem{PetzBook2008}
D.~Petz, \emph{Quantum Information Theory and Quantum Statistics}. Theoretical and Mathematical Physics, First ed., Springer; 2008.

%%%%%%%%%%%%%%%%%%%%%%%%%%%%%%%%%%%%%%%%%%%%%%%%%%%%%%%%%%%%%%%%%%%%%%%%%%%%%%%

\bibitem{GivensShortt1984}
C.R.~Givens, and R.M.~Shortt, ``A Class of Wasserstein Metrics for Probability Distributions". \emph{The Michigan Mathematical Journal}, Vol. 31, No. 2, pp. 231--240, 1984.


\bibitem{RachevRuschendorfBook1998}
S.T.~Rachev, and L.~R{\"u}schendorf, \emph{Mass Transportation Problems. Volume I: Theory}. First ed., Springer; 1998.

%%%%%%%%%%%%%%%%%%%%%%%%%%%%%%%%%%%%%%%%%%%%%%%%%%%%%%%%%%%%%%%%%%%%%%%%%%%%%%%

\bibitem{CarlenGangbo2003}
E.A.~Carlen, and W.~Gangbo, ``Constrained Steepest Descent in the 2-Wasserstein Metric". \emph{Annals of Mathematics}, pp. 807--846, 2003.
%\newpage

%%%%%%%%%%%%%%%%%%%%%%%%%%%%%%%%%%%%%%%%%%%%%%%%%%%%%%%%%%%%%%%%%%%%%%%%%%%%%%%%%%%%%%%

\bibitem{HalderWendelACC2016}
A.~Halder, and E.D.B.~Wendel, ``Finite Horizon Linear Quadratic Gaussian Density Regulator with Wasserstein Terminal Cost". \emph{Proceedings of the 2016 American Control Conference}, pp. 7249--7254, 2016.



%%%%%%%%%%%%%%%%%%%%%%%%%%%%%%%%%%%%%%%%%%%%%%%%%%%%%%%%%%%%%%%%%%%%%%%%%%%%%%%%%%%%%%%

\bibitem{HighamKimSIAM2001}
N.J.~Higham, and H.M.~Kim, ``Solving A Quadratic Matrix Equation by Newton's Method with Exact Line Searches". \emph{SIAM Journal on Matrix Analysis and Applications}, Vol. 23, No. 2, pp. 303--316, 2001.

%%%%%%%%%%%%%%%%%%%%%%%%%%%%%%%%%%%%%%%%%%%%%%%%%%%%%%%%%%%%%%%%%%%%%%%%%%%%%%%%%%%%%%%

\bibitem{LiberzonBrockettSIAM2000}
D.~Liberzon, and R.W.~Brockett, ``Spectral Analysis of Fokker--Planck and Related Operators Arising from Linear Stochastic Differential Equations". \emph{SIAM Journal on Control and Optimization}, Vol. 38, No. 5, pp. 1453--1467, 2000.

\bibitem{BrockettWillemsCDC1978}
R.W.~Brockett, and J.C.~Willems, ``Stochastic Control and the Second Law of Thermodynamics". \emph{Proceedings of the 1978 IEEE Conference on Decision and Control including the 17th Symposium on Adaptive Processes}, pp. 1007--1011, 1978.

%\bibitem{Brockett1997notes}
%\blue{\st{R.W.~Brockett, ``Notes on Stochastic Processes on Manifolds". \emph{Systems and Control in the Twenty-first Century}, pp. 75--100, Springer; 1997.}}

%%%%%%%%%%%%%%%%%%%%%%%%%%%%%%%%%%%%%%%%%%%%%%%%%%%%%%%%%%%%%%%%%%%%%%%%%%%%%%%%%%%%%%%

\bibitem{Stratonovich1960}
R.L.~Stratonovich, ``Application of the Theory of Markov Processes for Optimum Filtration of Signals". \emph{Radio Eng. Electron. Phys. (USSR)}, Vol. 1, pp. 1--19, 1960.

\bibitem{Kushner1964}
H.J.~Kushner, ``On the Differential Equations Satisfied by Conditional Densities of Markov Processes, with Applications". \emph{Journal of the SIAM Series A Control}, Vol 2, No. 1, pp. 106--119, 1964. 

%\bibitem{FrostKaliath1971PartIII}
%P.~Frost, and T.~Kaliath, ``An Innovations Approach to Least-squares Estimation--Part III: Nonlinear Estimation in White Gaussian Noise". \emph{IEEE Transactions on Automatic Control}, Vol. 16, No. 3, pp. 217--226, 1971.

\bibitem{FKK1972}
M.~Fujisaki, G.~Kallianpur, and H.~Kunita, ``Stochastic Differential Equations for the Non Linear Filtering Problem". \emph{Osaka Journal of Mathematics}, Vol. 9, No. 1, pp. 19--40, 1972.

\bibitem{KalmanBucy1961}
R.E.~Kalman, and R.S.~Bucy, ``New Results in Linear Filtering and Prediction Theory". \emph{Journal of Basic Engineering}, Vol. 83, No. 3, pp. 95--108, 1961. 



%%%%%%%%%%%%%%%%%%%%%%%%%%%%%%%%%%%%%%%%%%%%%%%%%%%%%%%%%%%%%%%%%%%%%%%%%%%%%%%%%%%%%%%

\bibitem{YezziVerriest2007}
A.~Yezzi, and E.I.~Verriest, ``Nonlinear Observers via Regularized Dynamic Inversion". \emph{Proceedings of the 2007 American Control Conference}, pp. 1693--1698, 2007.


\end{thebibliography}
\end{document}